\documentclass[reqno, 11pt]{amsart}
\usepackage{mathtools}
\usepackage{amsmath}
\usepackage{amssymb}
\usepackage{yhmath}
\usepackage{graphicx}
\usepackage{ mathrsfs }
\usepackage{bbm}
\usepackage{xcolor}
\usepackage{tikz-cd}
\usepackage{tikz}
\usetikzlibrary{patterns}
\usepackage{hyperref}

\DeclareMathAlphabet{\mathpzc}{OT1}{pzc}{m}{it}

\usepackage{thmtools}
\usepackage{thm-restate}

\usepackage{caption}

\newtheorem{theorem}{Theorem}[section]

\newtheorem*{claim*}{Claim}

\newtheorem{lemma}[theorem]{Lemma}
\newtheorem{lem}[theorem]{Lemma}
\newtheorem{corollary}[theorem]{Corollary}

\newtheorem{cor}[theorem]{Corollary}

\newtheorem{prop}[theorem]{Proposition}

\theoremstyle{definition}
\newtheorem{definition}[theorem]{Definition}
\newtheorem{Def}[theorem]{Definition}
\newtheorem{example}[theorem]{Example}

\theoremstyle{remark}
\newtheorem{remark}[theorem]{Remark}

\newtheorem{Rmk}[theorem]{Remark}

\numberwithin{equation}{section}

\newcommand{\norm}[1]{\lVert#1\rVert}
\newcommand{\op}{\operatorname}

\newcommand{\be}{\begin{equation}}
\newcommand{\ee}{\end{equation}}
\newcommand{\Ga}{\Gamma}
\newcommand{\bc}{\mathbb C}
\newcommand{\R}{\mathbb R}
\renewcommand{\H}{\mathbb H}
\newcommand{\Z}{\mathbb Z}

\newcommand{\ga}{\gamma}

\newcommand{\la}{\lambda}
\newcommand{\La}{\Lambda}
\newcommand{\inte}{\op{int}}
\newcommand{\ba}{\backslash}

\newcommand{\cal}{\mathcal}
\newcommand{\br}{\mathbb R}
\newcommand{\SO}{\op{SO}}
\newcommand{\Isom}{\op{Isom}}
\newcommand{\PSL}{\op{PSL}}
\newcommand{\F}{\cal F}

\newcommand{\bH}{\mathbb H}

\newcommand{\G}{\Gamma}

\renewcommand{\frak}{\mathfrak}

\newcommand{\e}{\varepsilon}
\newcommand{\BR}{\op{BR}}

\renewcommand{\L}{\mathcal L}
\newcommand{\fa}{\mathfrak a}

\renewcommand{\S}{\mathbb S}

\newcommand{\so}{\SO^\circ}

\newcommand{\Gr}{\Gamma_\rho}
\newcommand{\C}{\cal C}

\newcommand{\id}{\op{id}}

\newcommand{\ess}{\mathsf{E}}
\newcommand{\B}{\mathcal{B}}
\newcommand{\fg}{\frak g}
\renewcommand{\c}{\mathbb C}

\newcommand{\fk}{\mathfrak{k}}
\newcommand{\fp}{\mathfrak{p}}
\newcommand{\rank}{\op{rank}}
\newcommand{\Lie}{\op{Lie}}
\newcommand{\rd}{\mathfrak R_{\op{disc}}(\G, G_2)}
\newcommand{\nug}{\nu_{\op{graph}}}
\begin{document}

\title[Conformal measure rigidity]{Conformal measure rigidity for representations via self-joinings}

\author{Dongryul M. Kim}
\address{Department of Mathematics, Yale University, New Haven, CT 06511}
\email{dongryul.kim@yale.edu}

\author{Hee Oh}
\address{Department of Mathematics, Yale University, New Haven, CT 06511}
\email{hee.oh@yale.edu}
\thanks{
 Oh is partially supported by the NSF grant No. DMS-1900101.}
\begin{abstract}
Let $\Gamma$ be a Zariski dense discrete subgroup of a connected simple real algebraic group $G_1$. We discuss a rigidity problem for discrete faithful representations $\rho:\Gamma\to G_2$ and a surprising role played by higher rank conformal measures of the associated self-joining group.  Our approach recovers rigidity theorems of Sullivan, Tukia and Yue, as well as applies to Anosov representations, including Hitchin representations.

More precisely, for a given representation $\rho$ with a boundary map $f$ defined on the limit set $\Lambda$, we ask whether the extendability of $\rho$ to $G_1$ can be 
detected by the property that  $f$ pushes forward some $\Gamma$-conformal  measure class $[\nu_\Gamma]$ to a $\rho(\Gamma)$-conformal measure class $[\nu_{\rho(\Gamma)}]$.
When $\Gamma$ is of divergence type in a rank one group or when $\rho$ arises from an Anosov representation, we give an affirmative answer by showing that if the self-joining 
 $\Gamma_\rho=(\text{id} \times \rho)(\Gamma)$ is Zariski dense in $G_1\times G_2$, then the push-forward measures $(\text{id}\times f)_*\nu_\Ga$ and $(f^{-1}\times \text{id})_*\nu_{\rho(\Gamma)}$, which are higher rank $\Gamma_\rho$-conformal measures, cannot be in the same measure class.

\end{abstract}

\maketitle
\section{Introduction} 
For $i=1,2$, let $G_i$ be  a connected simple noncompact real algebraic group, and $(X_i, d_i)$ the  Riemannian symmetric space associated to $G_i$. 
Let $\Ga< G_1$ be a Zariski dense discrete subgroup. 
Let $$\rho:\Ga\to G_2 $$ be a discrete faithful Zariski dense representation.
In this paper, we are interested in the rigidity problem for $\rho$, that is,
 when can $\rho$ be extended to a Lie group isomorphism $G_1\to G_2$?
The rigidity of representations of discrete subgroups of a connected simple real algebraic group has been extensively studied in the last four decades. Especially 
when $\Ga$ is a lattice in $G_1$, which is not locally isomorphic to
$\PSL_2(\br)$, we have the celebrated Mostow strong rigidity when  $G_1=
G_2$ is of rank one and Margulis superrigidity when $G_1$ is of higher rank; in particular  any discrete faithful Zariski dense representation of $\Ga$ into $G_2$ extends to a Lie group isomorphism $G_1\to G_2$ (\cite{Mostowbook}, \cite{Prasad1973strong}, \cite{Margulisbook}). 

We will be interested in more general discrete subgroups, which are not necessarily lattices and which do admit non-trivial\footnote{We say $\rho:\Ga\to G_2$ is trivial if
it extends to a Lie group isomorphism $G_1\to G_2$} discrete faithful representations. The main aim of this paper is to present a new viewpoint in studying the rigidity problem for representations of discrete subgroups in terms of  conformal measure classes of self-joining groups.
For a given representation
$\rho:\Ga\to G_2$,
the self-joining of $\Ga$ via $\rho$ is the following discrete subgroup of $G_1\times G_2$:
$$\Ga_\rho:=(\id\times \rho)(\Ga) = \{ (\ga, \rho(\ga)) : \ga \in \Ga \}.$$
A basic observation is that $\Ga_\rho$ is not Zariski dense in $G_1\times G_2$ if and only if $\rho$ is trivial. This led us to search  properties of the
self-joining group $\Ga_\rho$ which would have direct implications on the rigidity of $\rho$; this was posed as a question in our earlier paper \cite{kim2022rigidity}. The main goal of this paper is to explain our findings
that  by studying  higher rank conformal measure classes of $\Ga_\rho$, we can relate the rigidity of $\rho$ with a relation among conformal measure classes of $\Ga$ and $\rho(\Ga)$ via its $\rho$-boundary map.

What is a $\rho$-boundary map? It is well-understood since Mostow's original proof of his strong rigidity theorem \cite{Mostowbook}
that investigating the behavior of 
{\it $\rho$ on the sphere at infinity} is important in the rigidity problem. For each $i=1,2$, let $\F_i $ denote the Furstenberg boundary of $G_i$, which is the quotient $G_i/P_i$ by a minimal parabolic subgroup $P_i$ of $G_i$.
We denote by $\La=\La_\Ga$ the limit set of $\Ga$,
defined as the set of all accumulation points of $\Gamma (o)$ in $\F_1$ for $o\in X_1$; it is the unique $\Ga$-minimal subset of $\F_1$ (Definition \ref{def.limitset}).
By a $\rho$-boundary map, we mean a $\rho$-equivariant continuous  embedding of $\La$ to $\F_2$. There can be at most one boundary map (Lemma \ref{ll}).
It will be our standing hypothesis that $\rho$ admits a boundary map
$$f:\La\to \F_2.$$

\subsection*{Conformal measures}
Our rigidity criterion involves the notion of conformal measures, which was introduced and studied extensively by Patterson and Sullivan for rank one discrete groups (\cite{Patterson1976limit}, \cite{Sullivan1979density}). This notion was generalized by Quint \cite{Quint2002Mesures}
to a discrete subgroup of any connected semisimple real algebraic group $G$ as follows.
Let $X$ be the associated symmetric space, $A$ a maximal diagonalizable subgroup with $\fa=\op{Lie} A$ and  $\F$ the Furstenberg boundary of $G$.
\begin{Def} Fix $o\in X$ and let $\Delta<G$ be a closed subgroup.
A Borel probability measure $\nu$ on ${\cal F}$ is called a $\Delta$-conformal measure
(with respect to $o$) if
there exists a linear form $\psi\in \fa^*$ such that for any $g\in \Delta$ and $\eta \in {\cal F}$, 
\be \label{eqn.ps0}
{ dg_* \nu \over d \nu} (\eta) = e^{\psi(\beta_{\eta}(o, g o))}
\ee 
where $g_* \nu(B) = \nu(g^{-1} B)$ for any Borel subset $B\subset \F$ and $\beta_{\cdot}(\cdot, \cdot)$ denotes the $\fa$-valued Busemann map (see Definition \ref{buse})\footnote{For $x\in X$, $d\nu_x (\eta) = e^{\psi( \beta_{\eta}(o, x))}d\nu(\eta)$
is a  $(\Delta, \psi)$-conformal measure with respect to $x$. The family $\{\nu_x:x\in X\}$ is referred to as a $(\Delta, \psi)$-conformal density. The uniqueness of a conformal measure is to be understood in terms of its associated conformal density.}.  We call $\psi$ a linear form associated to $\nu$. \end{Def}

Generalizing the construction of Patterson-Sullivan for rank one groups,
 Quint constructed a $\Delta$-conformal measure supported on the limit set $\La_\Delta\subset \F$ for any Zariski dense discrete subgroup $\Delta$ of $G$ \cite{Quint2002Mesures}.

\subsection*{Rigidity of $\rho$ and conformal measures}
We denote by
$$\mathfrak R_{\op{disc}}(\G, G_2)$$
the space of all discrete faithful Zariski dense representations
$\rho:\Ga\to G_2$ admitting boundary maps $f:\La\to \F_2$. If $\nu_\Ga$ is a $\G$-conformal measure supported on the limit set $\La$ of $\Ga$,
then the pushforward $f_*\nu_\Ga$ is a Borel measure supported on $f(\La)=\La_{\rho(\Ga)}$.
 If $\rho:\Ga\to G_2$ extends to a Lie group isomorphism $G_1\to G_2$, then
 $f_*\nu_\Ga$ is a $\rho(\Ga)$-conformal measure.

By investigating higher rank conformal measures for the associated self-joining group $\Ga_\rho$, we prove the converse holds:
if $[f_*\nu_\Ga]$ is a $\rho(\Ga)$-conformal measure class\footnote{For a given measure $\nu$, the notation $[\nu]$ denotes the class of all measures
equivalent to $\nu$, and we say that $[\nu]$ is a conformal measure class if
it contains a conformal measure.}
, 
then $\rho$ extends to a Lie group isomorphism
in the following two settings: 
\begin{enumerate}
    \item $\Ga$ is of divergence type in a rank one Lie group;
\item $\rho$ arises from an Anosov representation.
\end{enumerate}
We now describe our main theorems in these two settings.

\subsection*{Divergence-type groups}
We denote by $\text{rank } G_1$ the real rank of $G_1$, which is the dimension of a maximal real split torus of $G_1$.
When $\text{rank }G_1=1$, a discrete subgroup $\Ga<G_1$ is  {\it of divergence type} if
 $\sum_{\gamma \in \Ga} e^{-\delta_\Ga d_{1}(o_1, \ga o_1)} =\infty$ where $\delta_\Ga$ denotes  the critical exponent of $\Ga$.
 The divergence-type condition is satisfied for lattices, geometrically finite groups, as well as normal
 subgroups $\Ga$ of convex cocompact subgroups $\Ga_0$ with $\Ga_0/\Ga\simeq \Z^d$ for $d=1,2$
(\cite[Proposition 2]{Sullivan1984entropy}, \cite[Proposition 3.7]{Corlette1999limit}, \cite[Theorem 4.7]{Rees1981checking}).
Another important class of divergence-type groups are finitely generated discrete subgroups of $\PSL_2(\c)\simeq \SO^\circ(3,1)$ whose limit set is $\S^2 = \partial \H^3$; this follows from \cite[Corollary 11.2]{Canary1993ends} and the tameness theorem (\cite{agol2004tameness}, \cite{Calegari2006shrinkwrapping}).

For $\Ga$ of divergence type, there exists a unique $\Ga$-conformal measure of dimension $\delta_\Ga$, say, $\nu_\Ga$, supported on $\La$ (\cite{Sullivan1979density}, \cite[Corollary 1.8]{Roblin2003ergodicite}). 
\begin{theorem}\label{mmm} Let $\op{rank} G_1=1$ and $\Ga<G_1$ be of divergence type. Let $\rho\in \rd$.
Suppose  $\op{rank} G_2=1$ or $f:\La\to \F_2$ is a continuous extension of $\rho$.
Then  the following are equivalent:
\begin{enumerate}
\item  $[f_*\nu_\Ga]$  is a $\rho(\Ga)$-conformal measure class;
\item $\rho$ extends to a Lie group isomorphism $G_1\to G_2$.
\end{enumerate}
In particular, if  $G_1 \not\simeq G_2$ (e.g., $\op{rank}G_2\ge 2$), then  $f_*\nu_\Ga$ is not
equivalent to any $\rho(\Ga)$-conformal measure. 
\end{theorem}
That $f$ is a {\it continuous} extension of $\rho$ means that for any sequence $g_{\ell} \in \Ga$,
the convergence $g_\ell o_1\to
\xi$ implies the convergence $\rho(g_\ell) o_2\to f(\xi)$
in the sense of Definition \ref{conv}, where $o_1\in X_1$ and $o_2\in X_2$.
We refer to Example \ref{ex.bdr} of representations admitting boundary maps which are continuous extensions.

\subsection*{On the rigidity results of Mostow and Sullivan} 

The Mostow-Sullivan quasiconformal rigidity (\cite[Theorem 12.1]{Mostow1968quasiconformal}, {\cite[Theorem V]{Sullivan1981ergodic}}) can be deduced from Theorem \ref{mmm}:
\begin{cor} \label{c2}
\begin{enumerate}
\item Let $n\ge 3$. Any discrete subgroup $\Ga < \so(n, 1)$ of divergence type with $\La = \S^{n-1}$ is quasiconformally rigid.\footnote{That $\G$ is quasiconformally rigid means that any quasiconformal deformation of $\Ga$ in $\so(n,1)$ extends to a Lie group automorphism of $\so(n,1)$.}
\item Any finitely generated discrete subgroup $\Ga<\PSL_2(\c)$ with $\La=\S^2$ is  quasiconformally rigid.
\end{enumerate}
\end{cor}

 For $\Ga$ as in the above corollary, $\nu_\Ga$ is given by the Lebesgue measure $\op{Leb}_{\S^{n-1}}$.
On the other hand, for $n\ge 3$, a quasiconformal homeomorphism  $f:\S^{n-1}\to \S^{n-1}$ has
non-zero Jacobian at $\op{Leb}_{\S^{n-1}}$-almost every point \cite[Theorem 9.4]{Mostow1968quasiconformal}. Therefore $\op{Leb}_{\S^{n-1}}\ll f_*\op{Leb}_{\S^{n-1}}$\footnote{The notation $\nu_1\ll \nu_2$ means that $\nu_1$ is absolutely continuous with respect to $\nu_2$.}  and hence Corollary \ref{c2}(1) follows from Theorem \ref{mmm} (see also Lemma \ref{ergg}). If  $\Gamma<\SO^\circ (n,1)$ is a lattice, Mostow  \cite{Mostowbook} 
(see also Prasad {\cite[Theorem B]{Prasad1973strong}} for non-uniform lattices) showed that
 any discrete faithful representation of $\Ga$ into $\SO^\circ(n,1)$ is indeed a quasiconformal deformation, as the first step of his proof of rigidity. Hence  Mostow rigidity is a special case of Corollary \ref{c2}(1).
Since $\Ga$ as in (2) is of divergence type as mentioned before,  the case (2) is now a special case of~(1).

\begin{Rmk}\rm  When $X_1=X_2$ is the real hyperbolic space,
Theorem \ref{mmm} is due to Sullivan when $\delta_\Ga=\delta_{\rho(\Ga)}$ \cite[Theorem 5]{Sullivan1982discrete} and to Tukia in general \cite[Theorem 3C]{Tukia1989rigidity}.
 Yue extended it for a general rank one space \cite[Theorem A]{Yue1996mostow}. Their proofs use the ergodicity of the geodesic flow with respect to the Bowen-Margulis-Sullivan measure on the unit tangent bundle of $\Ga\ba X_1$. The rank one hypothesis on $G_2$ is  essential for their arguments.
 \end{Rmk} 
 
\subsection*{Anosov representations}
The notion of Anosov representations was introduced by Labourie for surface groups \cite{Labourie2006anosov}
and was extended by Guichard-Wienhard for hyperbolic groups \cite{Guichard2012anosov} (see also \cite{kapovich2018relativizing}). 
 For a hyperbolic group $\Sigma$, let $\rho_i : \Sigma \to G_i$ be a discrete faithful Zariski dense Anosov representation with respect to a minimal parabolic subgroup of $G_i$ for $i = 1, 2$. This means that
denoting by $\partial \Sigma$ the Gromov boundary,
there exists an equivariant homeomorphism $f_i:\partial \Sigma \to \La_{\rho_i(\Sigma)}\subset \F_i
$ and the limit set $\La_{\rho_i(\Sigma)}$ is antipodal, that is, two distinct points are in general position.

Setting $\Ga = \rho_1(\Sigma)$,  consider a representation of $\Ga$ given by $\rho := \rho_2 \circ \rho_1^{-1} : \Ga \to G_2$ with boundary map
$f=f_2\circ f_1^{-1}:\La_{\Ga} \to \F_2$. The following  may be regarded as an Anosov version of the rigidity theorem of
Sullivan-Tukia-Yue (\cite{Sullivan1982discrete}, \cite{Tukia1989rigidity}, \cite{Yue1996mostow}). 
\begin{theorem}\label{ananintro}\label{anan}
Let $\nu_\Ga$ be any $\Ga$-conformal measure on $\La_{\Ga}$.
 Either\footnote{The notation $\nu_1\perp\nu_2$ means that $\nu_1$ and $\nu_2$ are mutually singular to each other.}  $$f_* \nu_{\Ga}\perp \nu_{\rho(\Ga)}$$  for all $\rho(\Ga)$-conformal measures $\nu_{\rho(\Ga)}$ on $\La_{\rho(\Ga)}$
or $\rho : \Ga \to G_2$ extends to a Lie group isomorphism $G_1 \to G_2$.
\end{theorem}
This is equivalent to saying that either
$(f_1^{-1})_* \nu_{\Ga} \perp (f_2^{-1})_*\nu_{\rho(\Ga)}$ or there exists a
Lie group isomorphism $\Phi: G_1 \to G_2$ such that $\rho_2=\Phi\circ \rho_1$.

\subsection*{Hitchin representations} 
Let $\Gamma$ be  a torsion-free uniform lattice of $\PSL_2(\br)$,
and $\pi_d$ denote the $d$-dimensional irreducible representation   
$$\pi_d: \PSL_2(\br)\to \PSL_d(\br)$$ which is unique up to conjugation.
A Hitchin representation
$\rho:\Gamma\to \PSL_d(\br)$ is a representation which
 belongs to the same connected component as $\pi_d|_{\Gamma}$ in the character variety $\op{Hom}(\Gamma, \PSL_d(\br))/\sim$ where the equivalence is given by conjugations.
A Hitchin representation is known to be an Anosov representation with respect to a minimal parabolic subgroup by Labourie \cite{Labourie2006anosov}. Moreover, the Zariski closure
of its image is a connected simple real algebraic group (see \cite{Guichard}, \cite[Corollary 1.5]{SambarinoZc}). Therefore by considering $\rho$ as a representation into the Zariski closure of $\rho(\Gamma)$, we deduce the following corollary from Theorem \ref{anan}.
 The Furstenberg
boundary $\cal F$ of $\PSL_d(\br)$ is the full flag variety of $\br^d$.
Let $f:\S^1\to \F$ be the $\rho$-boundary map.
\begin{cor} \label{cor.hitchin}
For any Hitchin representation $\rho:\Gamma \to\PSL_d(\br)$,
$f_*{\op{Leb}_{\S^1}}$ is mutually singular to any $\rho(\Gamma)$-conformal measure on $f(\S^1)$, except for the case when $d=2$ and $\rho$ is a conjugation.
\end{cor}

\subsection*{Graph-conformal measures of self-joinings and proofs of main theorems}
As discussed before, the main novelty of our approach is the introduction of higher rank conformal measures for self-joinings in this rigidity problem.
Let $G=G_1\times G_2$, and recall the self-joining group:
$$\Ga_\rho=(\id\times \rho)(\Ga) = \{ (\ga, \rho(\ga))\in G : \ga \in \Ga \}.$$
The existence of a $\rho$-boundary map $f$ implies that the limit set of $\Ga_\rho$ is of the form $\La_\rho= ( \id\times f)(\La)=\{(\xi, f(\xi)):\xi\in \La\}$,  where $\id \times f:\La\to \La_\rho$ is the diagonal embedding $\eta \mapsto (\eta, f(\eta))$.

A general higher rank conformal measure seems mysterious, but the graph structure of the self-joining group $\Ga_\rho$ allows very explicit types of conformal measures, which we call graph-conformal measures. We write $\fa=\fa_1\oplus \fa_2$ where $\fa_i=\fa\cap \Lie G_i$ for $i=1,2$. Let $\pi_i:\fa\to \fa_i$ be the projection. A basic but crucial observation is that
for a $(\Ga,\psi_1)$-conformal measure $\nu_{\Ga,\psi_1}$ on $\La$,
 its pushforward measure by $\id \times f$:
$$ ( \id\times f)_*\nu_{\Ga,\psi_1}$$
is a $(\Ga_\rho, \psi_1\circ \pi_1)$-conformal measure, and conversely,
$(\Ga_\rho, \sigma)$-conformal measure on $\La_\rho$ for a linear form $\sigma\in \fa^*$ that factors through $\fa_1$ is of such form (Proposition \ref{lem.pushforward}).
We call them {\it graph-conformal}  measures of $\Ga_\rho$.

For a given $\Ga_\rho$-conformal measure $\nu$, the essential subgroup
$\mathsf E_{\nu}(\Ga_\rho)<\fa $ for $\nu$ consists of all
  $u\in \fa$ such that for any Borel subset $B \subset \F$ with $\nu(B) > 0$ and any $\varepsilon > 0$, there exists $g\in \Gr$ such that 
$$\nu (B \cap g B \cap \{ \xi \in \F : \| \beta_{\xi}(o,g o) - u\| < \varepsilon \} ) > 0.$$

Here is our key proposition linking the higher rank conformal measures and our rigidity question (see Proposition \ref{prop.loabscont}).
\begin{prop}\label{key} 
Let $\nu_1$ be a $(\Ga,\psi_1)$-conformal 
measure on $\La$ and $\nu_2$ a $(\rho(\Ga),\psi_2)$-conformal
measure on $\La_{\rho(\Ga)}$.
If $$\mathsf E_{(\id\times f)_*\nu_1}(\Ga_\rho) \not\subset \{(u_1, u_2)\in \fa: \psi_1(u_1)=\psi_2(u_2)\},$$ then
 $$(f^{-1}\times \id )_*(\nu_2)\not\ll  (\id\times f)_*(\nu_1)\text{ and hence }
\nu_2\not \ll f_*\nu_1 .$$
\end{prop}

The essential subgroup $\mathsf E_\nu(\Ga_\rho)$ is usually used as a tool
to decide the ergodicity of the corresponding Burger-Roblin measure $m^{\BR}_{\nu}$ 
(see Definition \ref{br}) for the maximal horospherical action  on $\Ga_\rho\ba G$ (\cite{Schmidt1977cocycles}, \cite{Roblin2003ergodicite}, \cite[Proposition 9.2]{lee2020invariant}).
In this paper, we use it as a tool to determine linear forms associated to
 conformal measures in the same measure class as $\nu$. In fact, recalling that $\Ga_\rho$ is Zariski dense in $G$
if and only if $\rho$ does not extend to a Lie group isomorphism $G_1\to G_2$, Theorem \ref{mmm} follows from Proposition \ref{key} and
the following dichotomy Theorem \ref{km}. 
For $\Ga$ and $\nu_\Ga$ as in Theorem \ref{mmm},
we set $$\nug=(\id\times f)_*\nu_\Ga,$$
which is the unique $(\Ga_\rho, \sigma_1)$-conformal measure where $\sigma_1(u_1, u_2)= \delta_\Ga u_1$ (Corollary \ref{cor.uniqueconformal}).
We refer to Theorem \ref{mplast} for any undefined terminologies:
\begin{theorem}\label{km} Let $\Ga, \rho$ be as in Theorem \ref{mmm}.
 In each of the two complimentary cases,
 the claims (1) to (4)  are equivalent to each other.

The first case is as follows:

 \begin{enumerate}
\item $\Ga_\rho$ is Zariski dense in $G$;
\item $\mathsf E_{\nug}(\Ga_\rho)= \fa$;
\item $m_{\nug}^{\BR}$ is $NM$-ergodic;
\item For any $(\Ga_\rho,\psi)$-conformal measure $\nu$
on $\La_\rho$ for $\psi\ne \sigma_1$, we have $[\nug]\ne [\nu]$.
\end{enumerate}

The second case is as follows:

 \begin{enumerate}
\item $\Ga_\rho$ is not Zariski dense in $G$;
\item $\fa = \R^2$ and $\mathsf E_{\nug}(\Ga_\rho)= \br  \cdot (\delta_{\rho(\Ga)}, \delta_{\Ga}) $;
\item $m_{\nug}^{\BR}$ is not $NM$-ergodic;
\item For any $(\Ga_\rho, \psi)$-conformal measure $\nu$
on $\La_\rho$ for a tangent linear form $\psi$, we have $[\nug]= [\nu]$.
\end{enumerate}

\end{theorem}

\begin{Rmk}  If $\rho$ is an Anosov representation, the essential subgroup for {\it any} $\Ga_\rho$-conformal measure on $\La_\rho$
is equal to $\fa$; this is a special case of  \cite[Proposition 10.2]{lee2020invariant}. 
In our general setting, we do not know
the size of a general essential subgroup. However for the essential subgroup
corresponding to the graph-conformal measure, we
are able to make an extensive use of the graph structures of $\Ga_\rho$
and $\La_\rho$
to prove $\mathsf E_{\nug}(\Ga_\rho)=\fa $.\end{Rmk} 

\subsection*{Cannon-Thurston map}
Let $\cal M$ be a closed hyperbolic 3-manifold  that fibers over a circle with
fiber a closed orientable  surface $S$.
Let $\Ga < \PSL_2(\R)$ be a uniform lattice such that $\Ga \ba \H^2$ is homeomorphic to $S$ by $\phi$. Let $\rho:\Ga\to \PSL_2(\bc)$ be the holonomy representation induced by
$\phi:\Ga\ba \bH^2\to S\subset \cal M$. Let $F:\bH^2\to \bH^3$ be a lift of $\phi$.
Cannon and Thurston \cite[Section 4]{Cannon2007peano} proved that $F$ continuously extends to the $\rho$-boundary map $f : \partial \H^2 \to \partial \H^3$, which is surjective but not injective. The map $f$ is called a Cannon-Thurston map.
Mj generalized this result for a general lattice $\Ga < \PSL_2(\R)$ (see \cite[Theorem 8.6]{Mj2014cannon}, \cite{Mj2018cannon}). 
Indeed, Theorem\ref{km} (see Theorem \ref{mplast}) applies to the Cannon-Thurston map, since we do not require $f$ to be injective in our proof of Theorem \ref{mplast}.
 Since  $\PSL_2(\br)$ is not isomorphic to $\PSL_2(\c)$, $\Ga_\rho$ is Zariski dense in $G=\PSL_2(\br)\times \PSL_2(\c)$.
Hence Theorem\ref{km} says that the Burger-Roblin measure $m_{\nug}^{\BR}$ is $NM$-ergodic where $\nug = (\id \times f)_* \op{Leb}_{\S^1}$, where $N<G$ is the product of strict upper triangular subgroups and 
$M=\{e\}\times \{\text{diag}(e^{i\theta}, e^{-i\theta})\}$. Therefore we have:
\begin{cor}\label{cannon} For Lebesgue almost all $\xi\in \S^1$,
the orbit $[g]NM$ with  $gP = (\xi, f(\xi))$ is dense
in the space $\{[h]\in \Ga_\rho\ba (\PSL_2(\br)\times \PSL_2(\c)): hP\in \La_\rho\}$.
\end{cor}

\subsection*{Organization} 
In section \ref{pre}, we review basic properties of Zariski dense discrete subgroups of semisimple real algebraic groups. In section \ref{sec.confess},
after recalling notions of conformal measures and essential subgroups, we discuss how the essential subgroup of a given conformal measure influences other conformal measures which are absolutely continuous with respect to $\nu$.
We also discuss the relation between essential subgroups and ergodic properties of  Burger-Roblin measures.
In section \ref{set}, we introduce the notion of graph-conformal measures for self-joining groups and explain their important role in the rigidity problem. We also prove the uniqueness of a $\rho$-boundary map.  In section \ref{sec.tangent}, we discuss tangent linear forms and show 
that self-joining subgroups admit infinitely many such forms.
In section \ref{swM}, we introduce the notion of weak-Myrberg limit set and show that it is of full measure with respect to graph-conformal measures.
In section \ref{sec.ess}, we prove that
the essential subgroups for graph-conformal measures are all of $\fa$ under the Zariski dense hypothesis on $\Ga_\rho$.
In section \ref{sec.proof}, we establish the dichotomy for the Zariski density of self-joining groups in terms of essential subgroups and singularity of conformal measures.
We then deduce the main theorems stated in the introduction.
In section \ref{sec.anosov}, we discuss general Anosov representations where
both $G_1$ and $G_2$ have no rank restrictions.

\section{Preliminaries} \label{sec.higherconf}\label{pre}
Let $G$ be a connected semisimple real algebraic group and $\fg=\op{Lie} G$ denote its Lie algebra. We fix a Cartan involution $\theta$ of $\mathfrak{g}$ and 
consider the decomposition $\mathfrak g=\mathfrak k\oplus\mathfrak{p}$, where $\fk$ and $\fp$ are the $+ 1$ and $-1$ eigenspaces of $\theta$, respectively. Denote by $K$ the maximal compact subgroup of $G$ with Lie algebra $\fk$. We also choose a maximal abelian subalgebra $\fa$ of $\mathfrak p$.
Fixing a left $G$-invariant and right $K$-invariant Riemannian metric on $G$ induces a Weyl-group invariant norm on $\mathfrak a$, which we denote by $\|\cdot\|$.  Note that the choice of this Riemannian metric induces a  left $G$-invariant metric $d$ on the symmetric space $X:=G/K$. Let $o\in X$ denote the point corresponding to the coset $[K]$. 

Let $A:=\exp \mathfrak a$ and choosing a closed positive Weyl chamber $\fa^+$ of $\fa$, set $A^+=\exp \mathfrak a^+$. Denote by $M$
the centralizer of $A$ in $K$. Set 
$N$ to be the maximal contracting horospherical subgroup for $A$: for an element
  $a$ in the interior of $A^+$,
  $N=\{g\in G: a^{-n} g a^n\to e\text{ as $n\to +\infty$}\}$. We set $ P=MAN,$  which is
a minimal parabolic subgroup of $G$. The quotient $$\F=G/P$$ is called the Furstenberg boundary of $G$, and is isomorphic to $K/M$. Two points $\xi, \eta \in \F$  are said to be in general position if
the diagonal orbit $G(\xi, \eta)$ is open in $\F\times \F$.

\begin{Def}[Busemann map] \label{buse} The Iwasawa cocycle $H: G\times \F \to \mathfrak a$ is defined as follows: for $(g, \xi)\in G\times \F$ with $\xi=[k]\in  K/M$,
$\exp H(g,\xi)$ is the $A$-component of $g k$ in the $KAN$ decomposition so that
$$gk\in K \exp (H(g, \xi)) N.$$
The Busemann map $\beta: \F\times X\times X\to\mathfrak a $ is now defined as follows: for $\xi\in \F$ and $[g], [h]\in X$,
 $$\beta_\xi ( [g], [h]):=H (g^{-1}, \xi)- H(h^{-1}, \xi).$$
\end{Def}

Observe that the Busemann map is continuous in all three variables. To ease notation, we will sometimes write  $\beta_\xi ( g, h)=\beta_\xi ( [g], [h])$. We have
$$\beta_\xi(g, h)+\beta_\xi (h, q)=\beta_\xi (g, q)\quad  \text{and}\quad 
\beta_{g\xi}(gh,gq) =\beta_\xi(h,q) $$
 for any $g, h, q \in G$ and $\xi \in \F$.

\subsection*{Jordan projection}
An element $g\in G$ is
loxodromic if $$g=h a m h^{-1}$$ for some $ a\in \inte A^+$,
$m\in M{}$ and $h\in G$. The Jordan projection of $g$
is defined to be $$\lambda(g):=\log a \in \inte \fa^+.$$ The attracting fixed point of $g$ is given by $$y_{g} := h P \in \F;$$ for any $\xi\in \cal F$ in general position with $y_{g^{-1}}$, the sequence $g^{\ell} \xi$ converges to $y_{g}$ as $\ell \to \infty$.

Let $\Delta<G$ be a discrete subgroup.  We write
$\lambda(\Delta)$ for the set of all Jordan projections of loxodromic elements of $\Delta$.
The following  result is due to Benoist \cite{Benoist2000proprietes}
(see \cite[Lemma 10.3]{lee2020invariant} for the second part).
\begin{theorem}\label{t1}
\label{dense0} 
If $\Delta<G$ is Zariski dense, then
$\lambda(\Delta)$ generates a dense subgroup of  $\fa$.
Moreover, for any finite subset $S\subset \lambda(\Delta)$,
$\lambda(\Delta)-S$ generates a dense subgroup of  $\fa$.
\end{theorem}

The limit cone $\L_{\Delta}\subset \fa^+$ is defined as the smallest closed
cone containing the Jordan projection of $\Delta$. If $\Delta$ is Zariski dense, $\L_\Delta$ is a convex cone with non-empty interior \cite{Benoist1997proprietes}.

\subsection*{Limit set}
Let $\Pi$ denote the set of all simple roots of $\mathfrak g  = \Lie G$ with respect to $\fa^+$.
 We say that a sequence $g_\ell \to \infty$ regularly in $G$  if  $\alpha(\mu(g_\ell)) \to\infty$ as $\ell\to\infty$ for all $\alpha\in\Pi$.

\begin{Def}\label{conv} A sequence $g_\ell \in G$, or $g_\ell o\in X$,
is said to converge to $\xi\in\cal F$ if
\begin{itemize}
\item $g_\ell \to\infty$ regularly in $G$;
\item 
$\lim\limits_{\ell \to\infty}\kappa_1(g_\ell)P=\xi$ where $\kappa_1(g_\ell)\in K$ is an element such that
$g_\ell\in \kappa_1(g_\ell) A^+K$.
\end{itemize}
In this case, we write $\xi=\lim_{\ell\to \infty} g_\ell =\lim_{\ell\to \infty} g_\ell o$.
\end{Def}

 \begin{Def}[Limit set] \label{def.limitset}
 The limit set $\La_{\Delta}\subset \F$ is defined as the set of all accumulation points
 of $\Delta (o)$ in $\cal F$, that is,
 $$\La_{\Delta}=\{\lim_{\ell \to \infty} g_\ell o \in \F:
 g_\ell \in \Delta\} .$$
 \end{Def}
This may be an empty set for a general discrete subgroup. However,
we have the following result of Benoist for Zariski dense subgroups (\cite[Section 3.6]{Benoist1997proprietes}, \cite[Lemma 2.15]{lee2020invariant}):

\begin{theorem}\label{t2}
\label{dense} 
If $\Delta<G$ is Zariski dense,
 then $\Lambda_\Delta$ is the unique $\Delta$-minimal subset of $\F$ and
the set of all attracting fixed points of loxodromic elements of $ \Delta$ is dense in
$\Lambda_\Delta$.
\end{theorem}

For each $g\in G$, there exists a unique element $\mu(g)\in \mathfrak a^+$, called the Cartan projection of $g$, such that
\begin{equation*}
g\in K\exp(\mu(g))K.
\end{equation*}

 \begin{definition}[$\fa^+$-valued distance]\label{under}\rm
We define $\underline a : X\times X\to\mathfrak a^+$ by 
$$\underline a(p,q):=\mu(g^{-1}h )$$
where $p=g(o)$ and $q=h(o)$.
\end{definition}

\subsection*{Growth indicator function}
Following Quint \cite{Quint2002divergence}, the growth indicator function $\psi_{\Delta} : \fa \to \R \cup \{ -\infty \}$ is defined as follows: for any cone $\cal C\subset \fa$, let
$\tau_{\C}$ denote the abscissa of convergence of the series $\sum_{g \in \Delta, \mu(g) \in \C} e^{-s \|\mu(g) \|}$.
For $u\in \fa-\{0\}$, we define
$$\psi_{\Delta} (u) = \| u \| \inf_{u \in \C} \tau_{\C}$$ where   the infimum is taken over all open cones $\C$ containing $u$. We also set $\psi_{\Delta}(0) = 0$. 
It is immediate from the definition that $\psi_{\Delta}$ is homogeneous of degree one.
Quint \cite[Theorem 1.1.1]{Quint2002divergence} showed that $\psi_{\Delta}$ is a concave and
upper semi-continuous function satisfying 
$$\L_\Delta=\{\psi_{\Delta}\ge 0\} .$$
Moreover $\psi_{\Delta} > 0$ on $\inte \L_{\Delta}$  and $\psi_\Delta =-\infty$ outside $\L_\Delta$.

  \section{Conformal measures and essential subgroups} \label{sec.confess}
Let $G$ be a connected semisimple real algebraic group. We continue notations $X, \F, o$, etc from section \ref{sec.higherconf}. Let $\Delta<G$ be a discrete subgroup.

\begin{Def}[Conformal measures] \label{ps} 
A Borel probability measure $\nu_o$ on $\F$ is called a $\Delta$-conformal measure (with respect to $o$)
if there exists a linear form $\psi\in \fa^*$ such that for all $\eta\in \F$ and $g\in \Delta$,
\be \label{co}
{ dg_* \nu_o \over d \nu_o}(\eta) = e^{\psi( \beta_{\eta}(o, g o))} .\ee 
In this case, we say $\nu_o$ is a $(\Delta,\psi)$-conformal measure.
For $x \in X$, $d\nu_x (\eta)= e^{\psi( \beta_{\eta}(o, x))}d\nu_o (\eta) $
is a $(\Delta, \psi)$-conformal measure with respect to $x$.
\end{Def}

The set of values $\{\beta_{\eta}(o, go)\in \fa: g\in \Delta, \eta\in \text{supp}(\nu_o)\} $ may not be large enough to determine the linear form to which $\nu_o$ is associated; in general, there may be multiple linear forms associated to the same conformal measure. This phenomenon occurs when $\dim \fa>1$ and hence definitely a higher rank feature which does not arise in rank one situation.

 \begin{lem}\label{ergg}
Let $\nu_1$ and $\nu_2$ be $\Delta$-quasi-invariant measures on $\F$ such that $\nu_1\ll \nu_2$. If $\nu_2$ is $\Delta$-ergodic, then
$[\nu_1]=[\nu_2].$
\end{lem}
 \begin{proof} 
 Let $B \subset \F$ be a $\nu_1$-null Borel subset. Since $\nu_1$
 is  $\Delta$-quasi-invariance, we have  $\nu_{1}(\Delta B) = 0.$ 
 Since $\nu_2$ is $\Delta$-ergodic and $\nu_1\ll \nu_2$, 
 it follows that $\nu_2(\Delta B)=0$.   Hence $\nu_2\ll \nu_1$.
 \end{proof}

\subsection*{Essential subgroups}
\begin{definition}[Essential subgroup for $\nu$]\label{ess}
   For a $\Delta$-conformal measure $\nu$ with respect to $o$, we define
   the subset $\mathsf E_{\nu}(\Delta) \subset \fa$ as follows:
  $u\in \mathsf E_{\nu}(\Delta)$  if for any Borel subset $B \subset \F$ with $\nu(B) > 0$ and any $\varepsilon > 0$, there exists $g\in \Delta$ such that 
\be \label{eqn.essdef}
\nu(B \cap g B \cap \{ \xi \in \F : \| \beta_{\xi}(o, g o) - u\| < \varepsilon \} ) > 0.
\ee
\end{definition}

It is easy to see that $\ess_{\nu}(\Delta)$ is a closed subgroup of the vector group $\fa$.
We call $\ess_{\nu}(\Delta)$ the essential subgroup for $\nu$.
In rank one case, this subgroup
was defined in (\cite{Schmidt1977cocycles}, see also \cite{Roblin2003ergodicite}) in order to study the ergodic properties of horospherical actions. The higher rank analogue given as above was studied in \cite{lee2020invariant} in relation with ergodicity of the maximal horospherical action with respect to the higher rank Burger-Roblin measures which we now recall.

\subsection*{Burger-Roblin measures in higher rank}

For $p= m(\exp a) n \in MAN$, set $H(p)=  a\in \fa$; in the notation of
Definition \ref{buse}, $H(p)=H(p, [P])=\beta_{[P]}(p^{-1}, e)$. 
\begin{Def}[Burger-Roblin measures] \label{br} Given a $(\Delta, \psi)$-conformal measure $\nu$ on $\F$ (with respect to $o$), we define  $\tilde m^{\BR}_\nu$ on $G$ as follows: for $\Phi\in C_c(G),$
$$\tilde m_\nu^{\BR}(\Phi)= \int_{kp\in KP} \Phi (k p) e^{\psi (H(p))} \,  d\nu(kP) dp $$
where  $dp$ denotes the left Haar measure on $P$. By the $\Delta$-conformal property of $\nu$ \eqref{co},  $\tilde m^{\BR}_\nu$  is left $\Delta$-invariant and hence induces a locally finite  measure on $\Delta\ba G$ \cite{edwards2020anosov}, which we denote by $m^{\BR}_\nu$. We call this measure
the Burger-Roblin (or BR) measure associated to $\nu$.
\end{Def}

The BR measures on $\Delta\ba G$ are $NM$-invariant and $A$-quasi-invariant measures (\cite[Lemma 3.9]{edwards2020anosov}).
The action of $NM$ on $\Ga_\rho\ba G$ will be referred to as a maximal horospherical action.
The $NM$-ergodicity of the BR-measure $m_\nu^{\BR}$ is
directly related to the size of the essential subgroup $\ess_{\nu}(\Delta)$ by the following higher rank version of a theorem of Schmidt (\cite{Schmidt1977cocycles}, see also \cite[Proposition 2.1]{Roblin2003ergodicite}).
\begin{prop}[{\cite[Proposition 9.2]{lee2020invariant}}] \label{prop.lo}  For any 
   $\Delta$-conformal ergodic measure $\nu$ on $\F$, 
   we have
    $$\ess_{\nu}(\Delta) = \fa  \mbox{ if and only if } (\Delta\ba G, m^{\BR}_\nu) \mbox{ is } NM\mbox{-ergodic.}$$
\end{prop}

\subsection*{Singularity of conformal measures by essential subgroups}
The following proposition is one of key ingredients of this paper:
\begin{prop}[{\cite[Proof of Lemma 10.21]{lee2020invariant}}] \label{prop.loabscont} For $i=1,2$, let $\nu_i$ be a
   $(\Delta, \psi_i)$-conformal measure for some $\psi_i \in \fa^*$.  If $\nu_2\ll \nu_1$,
   then 
   $$\psi_1(w)=\psi_2(w)\quad  \text{  for all $w\in \ess_{\nu_1}(\Delta)$}.$$

In particular, 
   if $\ess_{\nu_1}(\Delta) = \fa$,
   then $\nu_2 \ll \nu_1 $ implies
   $\psi_1=\psi_2. $
\end{prop}
\begin{proof} 
Suppose that
$\nu_{2}\ll\nu_{1}$. 
Consider the Radon-Nikodym derivative $\phi:=\frac{d\nu_{2}}{d\nu_{1}}\in L^1(  \F,\nu_{1})$.
Note that there exists a $\nu_{1}$-conull set $ F \subset \F$ such that for all $\xi\in F$ and $\ga\in\Delta$, we have
\begin{equation}\label{eq.FE1}
\phi (\ga^{-1}\xi)=e^{(\psi_2-\psi_1)(\beta_\xi(o,\ga o))}\phi (\xi).
\end{equation}

Choose $0<r_1<r_2$ such that
$$
B:=\{\xi\in  \F :r_1<\phi (\xi)<r_2\}
$$
has a positive $\nu_{1}$-measure.

 Suppose that there exists a vector $w\in \ess_{\nu_1}(\Delta) $ such that $\psi_1(w) \neq \psi_2(w)$.
Since $\ess_{\nu_1}(\Delta)$ is a subgroup of $\fa$, 
we have $\Z w\subset  \ess_{\nu_1}(\Delta) \cap  \{\psi_1\ne \psi_2\} $.
Hence we may assume by replacing $w$ with some element of $\Z w$ if necessary that
\begin{equation}\label{eq.ratio}
e^{(\psi_2-\psi_1)(w)}>\frac{2r_2}{r_1}.
\end{equation}

Choose $\e>0$ such that $e^{\norm{\psi_2-\psi_1}\e}<2$.
Since $\nu_{1}(B)>0$ and $w\in \ess_{\nu_1}(\Delta)$, there exists $\ga\in\Delta$ such that
$$
B':=B\cap\ga B\cap\{\xi\in  \F : \norm{\beta_\xi(o ,\ga o)-w}<\e\}
$$
has a positive $\nu_{1}$-measure.
Now note that
\begin{align*}
\int_{B'}\phi (\ga^{-1}\xi)\,d\nu_{1}(\xi)&>e^{(\psi_2-\psi_1)(w)-\norm{\psi_2-\psi_1}\e}\int_{B'}\phi (\xi)\,d\nu_{1}(\xi)\\
&>\frac{r_2}{r_1}\int_{B'}\phi (\xi)\,d\nu_{1}(\xi)
\end{align*}
by \eqref{eq.FE1}, \eqref{eq.ratio}, and the choice of $\e$.
In particular,
$$
\nu_{1} \left( \left\{\xi\in B': \phi(\ga^{-1}\xi)>\frac{r_2}{r_1} \phi(\xi)\right\}\right)>0.
$$
It follows that there exists $\xi\in B'\cap F$ such that
\be\label{bpr}
\phi (\ga^{-1}\xi)>\frac{r_2}{r_1}\phi (\xi).
\ee
On the other hand, for $\xi\in B'$, both $\xi$ and $\ga^{-1}\xi$ belong to $ B$.
By the definition of $B$, we have $\phi(\xi)>r_1$
and $\phi(\ga^{-1}\xi)<r_2$.
Therefore,
 for all $\xi\in B'$, we get
$$
\phi(\ga^{-1}\xi) < r_2 = {r_2 \over r_1} r_1 < {r_2 \over r_1} \phi(\xi).
$$
This is a contradiction to \eqref{bpr}.
\end{proof}

\section{Graph-conformal measures of
self-joinings}\label{set}
For $i = 1, 2$, let $G_i$ be a connected semisimple real algebraic group and let $(X_i, d_i)$ be the associated Riemannian symmetric space.
Let $(X, d)$ be the Riemannian product $(X_1\times X_2, \sqrt{d_1^2 + d_2^2})$. 
Set $$G=G_1\times G_2.$$
Then $G$ acts as isometries on $X$.
For $\square\in \{ A, M, N, P,K\}$, we consider the corresponding subgroups of $G$ by setting $$\square= \square_1\times \square_2.$$
In particular, $A=A_1\times A_2$. Let
$A^+=A_1^+ \times A_2^+$. Let $\fa$ denote the Lie algebra of $A$, and $\fa^+=\log A^+$. 
Set $\fa_i = \Lie A_i$ and $\fa_i^+ = \log A_i^+$ for $i = 1, 2$.
We note that $$\fa = \fa_1 \oplus \fa_2 $$
and that
every element of $\fa^+$ is the sum of elements of $\fa_1^+$ and $\fa_2^+$.

We also have $$\F = \F_1 \times \F_2$$
where $\F = G/P$ and $\F_i = G_i/P_i$ are Furtenberg boundaries for $i = 1, 2$.

Let $\Ga < G_1$ be a Zariski dense discrete subgroup.
Let $\rho:\Ga \to G_2$ be a discrete faithful Zariski dense representation. 

\begin{Def}[Self-joining] We define the self-joining of $\Ga$ via $\rho$ as  
$$\Ga_\rho:=(\id\times \rho)(\Ga) = \{ (g, \rho(g))\in G : g \in \Ga \},$$
which is a discrete subgroup of $G$.
\end{Def}
We begin by recalling the following:
\begin{lem}[{cf. \cite[Lemma 4.1]{kim2022rigidity}}] \label{Zdense}
If $G_1$ and $G_2$ are simple, the self-joining $\Ga_\rho < G$ is not Zariski dense  if and only if $\rho$ extends to a
Lie group isomorphism $G_1\to G_2$.
\end{lem}

In the rest of the paper, we assume that there exists a $\rho$-equivariant continuous map
$$f:\La \to \F_2$$
where $\La \subset \F_1$ is the limit set of $\Ga$.
We will not assume that $f$ is injective,
mentioned otherwise. When it is injective, we call it a $\rho$-boundary map.

\begin{example}\label{ex.bdr}
 Cases where a $\rho$-boundary map exists include the following: 
\begin{enumerate}
\item If $\Ga$ and $\rho(\Gamma)$ are geometrically finite and $\rho$ is type-preserving, the 
$\rho$-boundary map  exists and is a continuous extension of $\rho$; this was first shown by Tukia for $G_1=G_2=\so(n,1)$ \cite[Theorem 3.3]{Tukia1985isomorphisms} and generalized to all rank one groups by \cite[Theorem 0.1]{Yaman2004topological} and \cite[Theorem A.4]{Das2016tukia}.
\item 
If $\rho:\Ga\to \SO^\circ(n,1)$ is a quasiconformal deformation of $\Ga<\SO^\circ(n,1)$, i.e., there exists a quasiconformal homeomorphism $F:\S^{n-1}\to\S^{n-1}$ such that $\rho(\ga)= F\circ \ga\circ  F^{-1}$ for all $\ga\in \Ga$, then $f=F|_\La$
is the $\rho$-boundary map.

\item Let $\Ga<G_1$ be convex cocompact.
A Zariski dense representation $\rho:\Ga\to G_2$ is an {\it Anosov} representation (with respect to $P_2$) if there is a $\rho$-boundary map
$$f:\La \to \F_2$$ which maps two distinct points of $\La$ to points of $\F_2$ in general position.
Moreover, by the work of Kapovich-Leeb-Porti \cite[Theorem 1.4]{Kapovich2018morse}, $f$ is a continuous extension of~$\rho$.

\end{enumerate}
\end{example}

\subsection*{Uniqueness of the boundary map}
We will prove that there can be at most one $\rho$-boundary map (Lemma \ref{ll}).
First observe that, since  $\La$ (resp. $\La_{\rho(\Ga)}$) is 
the unique $\Ga$ (resp. $\rho(\Ga)$) minimal subset of $\F_1$ (resp. $\F_2$), it follows from the equivariance of $f$ that
$$f(\La)=\La_{\rho(\Ga)}.$$
Moreover,
$$\La_\rho=(\id \times f)(\Lambda) =\{(\xi, f(\xi)) \in \F_1\times \F_2:\xi\in \La \}$$
is the unique $\Ga_\rho$-minimal subset of $\F$; therefore
it is the limit set of $\Ga_\rho$ by Theorem \ref{dense}.

\begin{lem}[Uniqueness] \label{ll}
Let $\Ga < G_1$ be a Zariski dense discrete subgroup and $\rho : \Ga \to G_2$ a Zariski dense discrete faithful representation. If $ g \in \Ga$ and $\rho(g)$ are loxodromic, then
  $$f(y_{g})=y_{\rho(g)}.$$ 
  In particular, when $G_1$ and $G_2$ are simple, $f$ is the unique $\rho$-equivariant continuous map $\La\to \F_2$.
\end{lem} 
\begin{proof} Let $g\in \Ga$ be a loxodromic element.
Note that if $\xi\in \Lambda$ is in general position with $y_{g^{-1}}$, then, by the  $\rho$-equivariance and continuity of $f$,
 \be\label{rrr2}
 \rho(g)^{\ell} f(\xi) = f(g^{\ell} \xi) \to f(y_g)\quad\text{as $\ell\to +\infty$. }\ee 
On the other hand, if $\rho(g)$ is loxodromic and
$f(\xi)$ is in general position with $y_{\rho(g)^{-1}}$, then 
$\rho(g)^{\ell} f(\xi)  \to y_{\rho(g)}$ as $\ell\to +\infty$.

Let $\cal O_1$ (resp. $\cal O_2$) denote the set of points in $\F_1$ (resp. $\F_2$) which are in general position with $y_{g^{-1}}$ (resp. $y_{\rho(g)^{-1}}$). 
We claim that $\La\cap \cal O_1$ is dense in $\La$. Let $U$ be an open subset intersecting $\La$. Suppose
$U\cap \La \subset \F_1 - \cal O_1$. 
Since $\La$ is compact and $\Ga$ acts minimally on $\La$, $\La \subset \bigcup_{i = 1}^n \ga_i U$ for finitely many $\ga_1, \cdots, \ga_n \in \Ga$. We then have that $\La$ is contained in a union of finitely many proper subvarieties, $\bigcup_{i = 1}^n \ga_i  (\F_1 - \cal O_1)$. This contradicts the Zariski dense hypothesis of $\La$, proving that
 $\La \cap \cal O_1$ is dense in $\La$.

By the Zariski density of $\rho(\Ga)$ in $G_2$,  the limit set $f(\La) = \La_{\rho(\Ga)}$ is  Zariski dense in $\F_2$, and hence
 $\cal O_2 \cap f(\La)  \neq \emptyset$.
As $f$, regarded as a map $\La\to f(\La)$, is continuous,  $f^{-1}(\cal O_2\cap f(\La))$ is a non-empty open subset of $\La$. Therefore 
there exists $\xi_0 \in \La \cap \cal O_1  \cap f^{-1}(\cal O_2\cap f(\La))$.
Since $f(\xi_0) \in \cal O_2$, we have $\rho(g)^{\ell}f(\xi_0) \to y_{\rho(g)}$ as $\ell \to + \infty$. On the other hand, since $\xi_0 \in \cal O_1$, $\rho(g)^{\ell}f(\xi_0) \to f(y_g)$ as $\ell \to + \infty$ by \eqref{rrr2}. Therefore, $f(y_g) = y_{\rho(g)}$. This implies the first claim.

To prove the uniqueness, suppose that $f_1, f_2 : \La \to \F_2$ are $\rho$-equivariant continuous maps.
First consider the case when $\Gr$ is Zariski dense. By Theorem \ref{t2},   projecting $\La_{\rho} \subset \F_1 \times \F_2$ to its first factor $\La$, the set
$\La'=\{y_g: g \in \Ga \text{ loxodromic}, \rho(g)\text{ loxodromic}\}$ is dense in $\La$. By the first part of this lemma, $f_1=f_2$ on $\La'$ and hence by the continuity
of $f_1, f_2$, we get $f_1 = f_2$ on $\La$. 
In the case when $\Gr$ is not Zariski dense,  by Lemma \ref{Zdense}, $\rho$ extends to a Lie group isomorphism $G_1 \to G_2$ which must be an algebraic group isomorphism. Hence $\rho$ maps loxodromic element to a loxodromic element. Again,
the first part of this lemma implies that $f_1=f_2$ on $\La$.
\end{proof}

\subsection*{Graph-conformal measures} 
In the rest of the paper, for each $i=1,2$, we denote by
\be\label{pii} \pi_i:\fa\to \fa_i\ee 
the canonical projection.
For a linear form $\psi_i$ on $\fa_i^*$, we define
a linear form $\sigma_{ \psi_i}\in \fa^*$ by 
$$\sigma_{\psi_i}:=\psi_i\circ \pi_i;$$
so, $\sigma_{\psi_i}(u_1,u_2)= \psi_i(u_i)$ for all $(u_1, u_2)\in \fa_1\oplus \fa_2$. 
We fix $o=(o_1, o_2)\in X$.
\begin{prop} \label{lem.pushforward} Let $\psi_1\in \fa_1^*$.
\begin{enumerate}
    \item If $\nu_{\G,\psi_1}$ is a $(\Ga,\psi_1)$-conformal measure on $\La$ with respect to $o_1$,
then $(\id \times f)_*\nu_{\Ga, \psi_1} $ is a $(\Ga_\rho, \sigma_{\psi_1})$-conformal measure with respect to $o$.
 \item Any  $(\Ga_\rho, \sigma_{\psi_1})$-conformal measure on $\La_\rho$ with respect to $o$ is of the form $(\id \times f)_*\nu_{\Ga,\psi_1}$ 
for some  $(\Ga, \psi_1)$-conformal measure $\nu_{\Ga, \psi_1}$ on $\La$.
\end{enumerate}
\end{prop}

\begin{proof} For simplicity, we write $\nu_\Ga=\nu_{\Ga, \psi_1}$.
Clearly, the measure $(\id \times f)_*\nu_\Ga$ is supported on $\La_\rho$.
Since $\La_{\rho} = (\id \times f)(\La)$, any Borel subset
$\tilde E \subset \La_{\rho}$ is of the form $\tilde E= (\id \times f)(E)$ for some Borel subset $E \subset \Lambda$. Let  $\ga = (g, \rho(g))  \in \Gr$.
Then \be\begin{aligned} \label{eqn.pushforward1}
\ga_*(\id \times f)_*\nu_\Ga (\tilde E) & = (\id \times f)_*\nu_\Ga ( \ga^{-1}\tilde E) \\
& = (\id \times f)_* \nu_\Ga ( (\id \times f)(g^{-1} E)) \\
& = \nu_\Ga(g^{-1} E) = g_* \nu_\Ga (E).
\end{aligned}
\ee

Since $\nu_\Ga$ is a $(\Ga,\psi_1)$-conformal measure  with respect to
$o_1$, we have $$
\begin{aligned}
g_*\nu_\Ga (E) & = \int_E e^{\psi_1(\beta_{\eta}(o_1, g o_1))} d\nu_\Ga(\eta)\\
& = \int_{(\id \times f)(E)} e^{\psi_1( \beta_{\eta}(o_1, g o_1))} d (\id \times f)_*\nu_\Ga ((\id \times f)(\eta)).
\end{aligned}$$
Since the first component of $(\beta_{(\id \times f)(\eta)}(o, \ga o)) $ is
given by $\beta_{\eta}(o_1, g o_1)$, by performing a change of variable $\xi = (\id \times f)(\eta)$, we get  $$g_* \nu_\Ga(E) = \int_{\tilde E} e^{\sigma_{\psi_1} (\beta_{\xi}(o, \ga o))} d(\id \times f)_*\nu_\Ga(\xi).$$
Together with \eqref{eqn.pushforward1}, we obtain
$$\ga_*(\id \times f)_*\nu_{\Ga}(\tilde E) = \int_{\tilde E} e^{\sigma_{\psi_1}(\beta_{\xi}(o, \ga o))} d (\id \times f)_*\nu_{\Ga}(\xi).$$
Since $\ga\in \Gr$ is an arbitrary element, this proves $(\id \times f)_*\nu_{\Ga}$ is a $(\Ga_\rho, \sigma_{\psi_1})$-conformal measure with respect to $o$. This proves (1).

To prove (2), let $\nu$ be a $(\Gr, \sigma_{\psi_1})$-conformal measure on $\La_\rho$ with respect to $o$. We denote by $\pi : \F \to \F_1$ the canonical projection and claim that the pushforward $\pi_*\nu$ is a $(\Ga,\psi_1)$-conformal measure on $\La$
 with respect to $o_1$. To see this,
consider any $g \in \Ga$ and any Borel subset $E \subset \La$.
Then $$g_* \pi_* \nu(E) = \pi_*\nu (g^{-1}E) = \nu(g^{-1}E \times \F_2).$$ Since $\nu$ is supported on $\La_{\rho} = (\id \times f)(\La)$ and $f$ is $\rho$-equivariant, we have
$$g_* \pi_* \nu(E) = \nu((\id \times f)(g^{-1}E)) = (g, \rho(g))_* \nu((\id \times f)(E)).$$
By the $(\Gr, \sigma_{\psi_1})$-conformality of $\nu$,  we have  $$g_* \pi_* \nu(E) = \int_{(\id \times f)(E)} e^{\psi_1( \beta_{\pi(\xi)}(o_1, g o_1))} d\nu(\xi) = \int_{E} e^{\psi_1 ( \beta_{\eta}(o_1, g o_1))} d \pi_*\nu(\eta).$$
Hence we have $${dg_* \pi_* \nu \over d \pi_* \nu}(\eta) = e^{\psi_1( \beta_{\eta}(o_1, g o_1))},$$ 
proving that $\pi_*\nu$ is a $(\Ga, \psi_1)$-conformal measure on $\La$ with respect to $o_1$. Set $\nu_{\Ga, \psi_1}=\pi_*\nu$.
Now for any Borel subset $E \subset \La$, 
$$\nu((\id \times f)(E)) = \pi_*\nu(E) = \nu_{\Ga,\psi_1}(E) = (\id \times f)_*\nu_{\Ga, \psi_1} ((\id \times f)(E)).$$ 
Since both $\nu$ and $ (\id \times f)_*\nu_{\Ga, \psi_1}$ are supported on $\La_{\rho} = (\id \times f)(\La)$, this proves that
$\nu=(\id \times f)_*\nu_{\Ga, \psi_1}$, finishing the proof.
\end{proof}

We have the following corollary:
\begin{cor} \label{cor.uniqueconformal} 
If $\nu_{\Ga, \psi_1}$ is a unique $(\Ga, \psi_1)$-conformal measure on $\La$,
then $(\id \times f)_*\nu_{\Ga,\psi_1}$ is the unique $(\Ga_\rho, \sigma_{\psi_1})$-conformal measure on $\La_\rho$
with respect to $o$; in particular,
$(\id \times f)_*\nu_{\Ga,\psi_1}$ is $\Ga_\rho$-ergodic.
\end{cor}
\begin{proof}
   The first claim is clear from  Proposition \ref{lem.pushforward}.
   Ergodicity of $ (\id \times f)_*\nu_{\Ga, \psi_1}$  follows immediately from the uniqueness.
\end{proof}

\begin{Def}[Graph-conformal measures] By
a graph-conformal measure of $\Ga_\rho$, we mean a (conformal) measure
of the form $(\id \times f)_*\nu_{\Ga}$ for some $\Ga$-conformal measure
$\nu_\Ga$ on $\La$.
\end{Def}

Using this terminology, Proposition \ref{lem.pushforward} can be reformulated as follows:
\begin{prop}\label{gc}
 Let $\sigma\in \fa^*$ be a linear form which factors through $\fa_1$.
A
$(\Ga_\rho, \sigma)$-conformal measure on $\La_\rho$  is
a graph-conformal measure of $\Ga_\rho$, and conversely, any graph-conformal measure of $\Ga_\rho$ is of such a form.
    \end{prop}

\begin{lem}\label{s2} Suppose $f$ is injective. 
 Let  $\nu_{\Ga, \psi_1}$ and
 $\nu_{\rho(\G),\psi_2}$ be $\Ga$-conformal and $\rho(\Ga)$-conformal measures respectively.

Then $(f^{-1}\times \id )_*\nu_{\rho(\Ga), \psi_2} $ is a $(\Ga_\rho, \sigma_{\psi_2})$-conformal measure, and we have
$$\text{$(f^{-1}\times \id)_*\nu_{\rho(\Ga), \psi_2}\ll(\id\times f)_*\nu_{\Ga,\psi_1}$ if and only if $\nu_{\rho(\Ga),\psi_2}\ll f_*\nu_{\Ga,\psi_1} $}$$
and
$$[(f^{-1} \times \id)_* \nu_{\rho(\Ga), \psi_2}] = [(\id \times f)_* \nu_{\Ga, \psi_1}] \mbox{ if and only if } [\nu_{\rho(\Ga), \psi_2}] = [f_*\nu_{\Ga, \psi_1}].$$

\end{lem}
\begin{proof} The first claim can be proved similarly as (1) of Proposition \ref{lem.pushforward}.
The second claim follows since any Borel subset $\tilde E$ of $\La_\rho$  is of the form $(\id \times f)(E)=(f^{-1}\times \id)(f(E))$ for a Borel subset $E\subset \La$. Moreover
 $$(\id\times f)_*\nu_{\Ga,\psi_1}(\tilde E) = \nu_{\Ga,\psi_1}(E) \quad \mbox{and} \quad (f^{-1}\times \id)_*\nu_{\rho(\Ga),\psi_2}(\tilde E) = \nu_{\rho(\Ga),\psi_2}(f(E)),$$
as indicated in the following commutative diagram:
$$\begin{tikzcd}
& \F & \\
\La \arrow[ru, "\id \times f"] \arrow[rr, "f"'] & & \La_{\rho(\Ga)} \arrow[lu, "f^{-1} \times \id"']
\end{tikzcd} $$
\end{proof}

\section{Tangent forms for self-joinings} \label{sec.tangent}
For a discrete subgroup $\Delta$ of a semisimple real algebraic group $G$, the growth indicator function $\psi_\Delta$ defined in section \ref{pre} plays an important role in studying $\Delta$-conformal measures. For instance, if there exists a $(\Delta, \psi)$-conformal measure, then $\psi\ge \psi_\Delta$ \cite[Theorem 8.1]{Quint2002Mesures}.

\begin{Def} A linear form $\psi\in \fa^*$  is said to be
tangent to $\psi_\Delta$, or simply a tangent form,
if  $\psi \ge \psi_{\Delta}$ and $\psi(u) = \psi_{\Delta}(u)$ for some $u\in \fa^+ - \{0\}$; note that $u$ must belong to $\L_\Delta$, as $\psi_\Delta=-\infty$ outside
$\L_\Delta$. When we would like to emphasize the role of $u$, we will say that $\psi$ is tangent to $\psi_\Delta$ at $u$.
\end{Def}

For any linear form $\psi$ tangent to $\psi_\Delta$ at a direction of $\inte \fa^+$,
Quint \cite[Theorem 8.4]{Quint2002Mesures} constructed a $(\Delta, \psi)$-conformal measure supported on the limit set $\La_{\Delta}$, generalizing Patterson-Sullivan's construction.

For instance, if $\Delta$ is a lattice in $G$, $\psi_\Delta$ is equal to the sum of all positive roots (counted with multiplicity)  on $\fa^+$ and hence there is only one tangent form at $\inte \fa^+$, which
is $\psi_\Delta$ itself and the corresponding conformal measure
is simply the $K$-invariant probability measure on $\F$. 
In contrast, we show in this section that there are infinitely many tangent forms for self-joinings. 

\subsection*{Tangent forms for self-joinings} We use the same notations as in section \ref{set}.
Let $G_1$ and $ G_2$ be connected semisimple real algebraic groups.
Let $\Ga < G_1$ be a Zariski dense discrete subgroup and
$\rho:\Ga \to G_2$ a discrete faithful Zariski dense representation.
 Let $\L_{\rho}$ and $\psi_{\rho}$ denote the limit cone and the growth indicator function of the self-joining $\Ga_{\rho} < G=G_1\times G_2$ respectively. 

The main goal of this section is to prove the following:
\begin{theorem} \label{thm.manycritical}
If $\Gr$ is Zariski dense in $G$, there are infinitely many linear forms which are tangent to $\psi_\rho$  at directions of $ \inte \fa^+$. 
\end{theorem}

\begin{corollary} \label{cor.manyconformal}
For any finite collection $\varphi_1, \cdots, \varphi_n \in \fa^*$, there exists a $(\Gr, \psi)$-conformal measure on $\La_{\rho}$ for some linear form $\psi \in \fa^*-\{ \varphi_i: i = 1, \cdots, n\}$. 
\end{corollary}

\begin{proof}
By Theorem \ref{thm.manycritical}, we have a  linear form $\psi \notin \{\varphi_1, \cdots, \varphi_n\}$ tangent to
$\psi_\rho$ at $\inte \fa^+$. By \cite[Theorem 8.4]{Quint2002Mesures}, there exists a $(\Gr, \psi)$-conformal measure supported on $\La_{\rho}$.
\end{proof}

For a linear form $\psi\in \fa^*$, we let
$-\infty\le \delta_\psi\le\infty $ denote the abscissa of convergence of the series 
$\sum_{g\in \Delta} e^{-s \psi(\mu(g))} $. 
We will use the following lemma in the proof of Theorem \ref{thm.manycritical}.
\begin{lem}[{\cite[Theorem 2.5]{kim2021tent}}] \label{s1} Let $\Delta<G$ be a Zariski dense discrete subgroup.
For any linear form $\psi\in \fa^*$ such that $\psi|_{\L_\Delta}\ge 0$
and  $\delta_\psi<\infty$, the linear form
$\delta_\psi \psi \in \fa^*$ is tangent to $\psi_\Delta$. 
\end{lem}
\medskip
\noindent{\bf Proof of Theorem \ref{thm.manycritical}.}
For convenience, we let $\mathsf C_\rho$ be the space of all linear forms tangent to $\psi_\rho$ at $\inte \fa^+$. Since $\psi_\rho$ is concave,
the subset
$S:=\{(u, t) : u \in \L_{\rho}, 0 \le t \le \psi_{\rho}(u)\}$ 
is a convex subset. Hence for each $v \in \inte \L_{\rho}$, we may
apply the supporting hyperplane theorem\footnote{The supporting hyperplane theorem says that for any convex subset $S \subset \R^n$ and $x_0 \in \partial S$, there exists a hyperplane $H \subset \R^n$ containing $x_0$ such that  $S$ is contained in a closed half-space bounded by $H$ (cf. \cite[Corollary IV.3.4]{Conway_functional}).} to $S$
 to get a linear form $\psi_v \in \fa^*$ tangent to $\psi_\rho$ at $v$. Hence $\psi_v \in \mathsf C_{\rho}$ and  $$\psi_{\rho}(v) = \min_{\psi \in \mathsf C_{\rho}} \psi(v).$$

Since $\Gr$ is Zariski dense, we have $\inte \L_{\rho}\ne \emptyset$.
Choose unit vectors $\mathsf{y} \in \fa_1^+$ and $\mathsf{x} \in \fa_2^+$ so that the line segment $\mathsf{A} = [\mathsf x,\mathsf y]=
\{ (1-t)\mathsf{x} + t\mathsf{y} : 0 \le t \le 1\}$ intersects $\inte \L_{\rho}$. Since $\L_{\rho}$ is convex,  the intersection $\mathsf{A} \cap \inte \L_{\rho}$ is an open line segment. Let $J$ be the closure of $\mathsf{A} \cap \inte \L_{\rho}$.
We write $J=[v_1, v_2]$ where  $v_1
=(1-t_1)\mathsf x+t_1 \mathsf y$ and $ v_2=(1-t_2) \mathsf x+t_2 \mathsf y\in \mathsf A$ are
two endpoints of $J$ and $t_1<t_2$.

We now suppose that \be \label{cc} \# \mathsf C_{\rho} < \infty.\ee 
Since $\psi_{\rho}$ is concave and upper semi-continuous, it is continuous on the interval $J$. Hence the finiteness assumption on $\mathsf{C}_{\rho}$ implies that there exists a finite partition of $J=[v_1, v_2]$ into $v_1 = w_0<w_1<\cdots < w_n=v_2$  and 
pairwise distinct $$\psi_1, \cdots, \psi_n \in \mathsf C_{\rho}$$  such that
$\psi_\rho=\psi_\ell$ on each $[w_{\ell-1}, w_{\ell}]$ for $1\le \ell \le n$.

\medskip
{\noindent \bf Claim (1): $n = 1$.}
If  $n \ge 2$, then $w_1 \in \inte J$, and hence we find infinitely many elements of $\mathsf C_\rho$ by considering convex linear combinations
$$(1-c)\psi_1 + c\psi_2 $$ for all $0 < c < 1$; hence $n=1$.

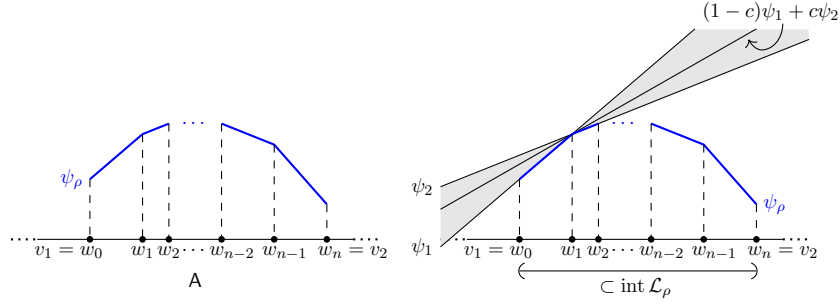
\begin{figure}[h]
\begin{tikzpicture}[scale=0.7, every node/.style={scale=0.7}]

        \draw[white] (-2, -0.6) -- (2.5, -0.6);
        \draw[white] (-1.9, -0.45) .. controls (-2.05, -0.5) and (-2.05, -0.7) .. (-1.9, -0.75);
        \draw[white] (2.4, -0.45) .. controls (2.55, -0.5) and (2.55, -0.7) .. (2.4, -0.75);
        \draw[white] (0.2, -0.6) node[below] {$\subset \inte \L_{\rho}$};

        \draw[thick, dotted] (-3.5, 0) -- (-3, 0);
        \draw[thick, dotted] (3, 0) -- (3.5, 0);
        \draw (-3, 0) -- (3, 0);
        \draw (0, -0.5) node[below] {$\mathsf A$};


        \draw[thick, blue] (-2, 8/7) -- (-1, 2)  -- (-0.5, 2.2);
        \draw[blue] (0, 2.2) node {$\cdots$};
        \draw[thick, blue]  (0.5, 2.2) -- (1.5, 1.8) -- (2.5, 2/3);
        \draw[blue] (-2, 8/7) node[left] {$\psi_{\rho}$};

        \draw[dashed] (-2, 8/7) -- (-2, 0);
        \filldraw (-2, 0) circle(1.5pt);
        \draw (-2.4, 0) node[below] {$v_1 = w_0$};

        \draw[dashed] (-1, 2) -- (-1, 0);
        \filldraw (-1, 0) circle(1.5pt);
        \draw (-1, 0) node[below] {$w_1$};
        
        \draw[dashed] (-0.5, 2.2) -- (-0.5, 0);
        \filldraw (-0.5, 0) circle(1.5pt);
        \draw (-0.5, 0) node[below] {$w_2$};

        \draw (0, 0) node[below] {$\cdots$};
        
        \draw[dashed] (0.5, 2.2) -- (0.5, 0);
        \filldraw (0.5, 0) circle(1.5pt);
        \draw (0.7, 0) node[below] {$w_{n-2}$};
        
        \draw[dashed] (1.5, 1.8) -- (1.5, 0);
        \filldraw (1.5, 0) circle(1.5pt);
        \draw (1.7, 0) node[below] {$w_{n-1}$};
        
        \draw[dashed] (2.5, 2/3) -- (2.5, 0);
        \filldraw (2.5, 0) circle(1.5pt);
        \draw (3, 0) node[below] {$w_{n}=v_2$};

    \end{tikzpicture}
    \begin{tikzpicture}[scale=0.7, every node/.style={scale=0.7}]

        \draw[thick, dotted] (-3.5, 0) -- (-3, 0);
        \draw[thick, dotted] (3, 0) -- (3.5, 0);
        \draw (-3, 0) -- (3, 0);

        \draw (-2, -0.6) -- (2.5, -0.6);
        \draw (-1.9, -0.45) .. controls (-2.05, -0.5) and (-2.05, -0.7) .. (-1.9, -0.75);
        \draw (2.4, -0.45) .. controls (2.55, -0.5) and (2.55, -0.7) .. (2.4, -0.75);
        \draw (0.2, -0.6) node[below] {$\subset \inte \L_{\rho}$};



        \draw[draw=white, fill=gray!20, opacity=0.5] (-3.5, -1/7) -- (-1, 2) -- (-3.5, 1);
        \draw[draw=white, fill=gray!20, opacity=0.5] (-1, 2) -- (4/3, 4) -- (3.5, 4) -- (3.5, 3.8);

        \draw (-3.5, -1/7) -- (4/3, 4);
        \draw (-3.5, -1/7) node[left] {$\psi_1$};
        
        \draw (-3.5, 1) -- (3.5, 3.8);
        \draw (-3.5, 1) node[left] {$\psi_2$};

        \draw (-3.5, 4/7) -- (5/2, 4);
        \draw (2.775, 4) node[above] {$(1-c)\psi_1 + c \psi_2$};
        \draw[->] (3, 4.1) .. controls (3, 3.8) and (2.8, 3.6) .. (2.6, 3.6) .. controls (2.5, 3.6) and (2.4, 3.7) .. (2.35, 3.8); 

        \draw[thick, blue] (-2, 8/7) -- (-1, 2)  -- (-0.5, 2.2);
        \draw[blue] (0, 2.2) node {$\cdots$};
        \draw[thick, blue]  (0.5, 2.2) -- (1.5, 1.8) -- (2.5, 2/3);
        \draw[blue] (2.5, 2/3) node[right] {$\psi_{\rho}$};

        \draw[dashed] (-2, 8/7) -- (-2, 0);
        \filldraw (-2, 0) circle(1.5pt);
        \draw (-2.4, 0) node[below] {$v_1 = w_0$};

        \draw[dashed] (-1, 2) -- (-1, 0);
        \filldraw (-1, 0) circle(1.5pt);
        \draw (-1, 0) node[below] {$w_1$};
        
        \draw[dashed] (-0.5, 2.2) -- (-0.5, 0);
        \filldraw (-0.5, 0) circle(1.5pt);
        \draw (-0.5, 0) node[below] {$w_2$};

        \draw (0, 0) node[below] {$\cdots$};
        
        \draw[dashed] (0.5, 2.2) -- (0.5, 0);
        \filldraw (0.5, 0) circle(1.5pt);
        \draw (0.7, 0) node[below] {$w_{n-2}$};
        
        \draw[dashed] (1.5, 1.8) -- (1.5, 0);
        \filldraw (1.5, 0) circle(1.5pt);
        \draw (1.7, 0) node[below] {$w_{n-1}$};
        
        \draw[dashed] (2.5, 2/3) -- (2.5, 0);
        \filldraw (2.5, 0) circle(1.5pt);
        \draw (3, 0) node[below] {$w_{n}=v_2$};
    \end{tikzpicture}
    \caption{Partition of $J$ (left) and when $n \ge 2$ (right)} \label{fig.trivialpartition}
\end{figure}

\medskip
{\noindent \bf Claim (2): $J = \mathsf{A}$.}
This claim is the same as $v_1, v_2 \notin \inte \fa^+$. Suppose that $v_i \in \inte \fa^+$. Since $v_i \in \partial \L_{\rho}$, it follows from the supporting hyperplane theorem applied to $\L_{\rho}$ that there exists a hyperplane $\sf P$ tangent to $\L_{\rho}$ at $v_i$. Since $\L_{\rho}$ is a closed cone in $\fa^+$, $\sf P$ is a linear subspace. Hence there exists a linear form $\varphi \in \fa^*$ such that $\ker \varphi = \mathsf P$ and $\varphi \ge 0$ on $\L_{\rho}$. In particular, $\varphi > 0$ on $\inte \L_{\rho}$. Then for any $c > 0$, the linear combination $\psi_1 + c \varphi$ belongs to $ \mathsf C_{\rho}$, yielding a contradiction to \eqref{cc}.

\medskip
{\noindent \bf Claim (3): $\psi_{\rho} = 0$ on $J = \mathsf{A}$.} By Claims (1) and (2), $\psi_{\rho} = \psi_1$ on $\mathsf{A}$. 
Since the growth indicator function
$\psi_{\Ga}$ is upper semi-continuous, it attains a maximum on the unit sphere of $\fa_1$. Hence we may choose $\Phi \in \fa_1^*$ such that
 $\Phi > \psi_{\Ga}$ on $\fa_1 - \{0\}$. By \cite[Lemma 3.1.3]{Quint2002divergence}, we have $\delta_{\Phi}<\infty$.
We now set $\tilde \varphi_1:=\Phi\circ \pi_1$. Since $\tilde \varphi_1 (\mu((g, \rho(g)))) = \Phi (\mu(g))$ for $g \in \Ga$, we have $$\delta_{ \tilde \varphi_1} = \delta_{\Phi} < \infty.$$ 
Since $\pi_1(\L_{\rho}) \subset \fa_1^+$ and $\Phi > 0$ on $\fa_1^+-\{0\}$, we have
$\tilde \varphi_1\ge 0$ on $\L_\rho$. 
Hence by Lemma \ref{s1}, $\varphi_1 := \delta_{\tilde \varphi_1} \tilde {\varphi}_1$ is tangent to $\psi_{\rho}$; in particular, 
\be \label{eqn.sigma1}
\psi_{\rho} \le \varphi_1 \quad \mbox{and} \quad \varphi_1|_{\fa_2} = 0.
\ee
Similarly, we have a linear form $\varphi_2 \in \fa^*$ which is tangent to $\psi_{\rho}$ and factors through  $\pi_2 : \fa \to \fa_2$. It implies \be \label{eqn.sigma2}
\psi_{\rho} \le \varphi_2 \quad \mbox{and} \quad \varphi_2|_{\fa_1} = 0.
\ee
 Since $\mathsf x \in \L_{\rho} \cap \mathsf A \cap \fa_2$ and $\mathsf y \in \L_{\rho} \cap \mathsf A \cap \fa_1$,
we deduce from  \eqref{eqn.sigma1} and \eqref{eqn.sigma2}:
$$
0 \le \psi_1(\mathsf x) = \psi_{\rho}(\mathsf x) \le \varphi_1(\mathsf x) = 0 \quad \mbox{and} \quad  0  \le \psi_1(\mathsf y) = \psi_{\rho}(\mathsf y) \le \varphi_2( \mathsf y) = 0
.$$ 
Since $\psi_1$ is a linear form and $\mathsf{A} = [\mathsf x,\mathsf y]$, we have $$\psi_{\rho} = \psi_1 = 0 \mbox{ on } \mathsf{A}.$$

\medskip
{\noindent \bf Finishing the proof.} Since $\mathsf A$ intersects $\inte \L_{\rho}$ and $\psi_{\rho} > 0$ on $\inte \L_{\rho}$ by \cite[Theorem 1.1.1]{Quint2002divergence}, Claim (3) yields a contradiction. Therefore, $\# \mathsf C_{\rho} = \infty$, finishing the proof.

\section{Weak-Myrberg limit sets of self-joinings}\label{swM}
Let $G_1$ be a connected simple real algebraic group of rank one and $(X_1, d_1)$ the associated rank one symmetric space. Let $\Ga < G_1$ be a non-elementary  discrete subgroup with limit set $\La = \La_\Ga \subset \F_1$.

\subsection*{Divergence-type groups} Fix $o_1\in X_1$.
If  the Poincar\'e series of $\Ga$ diverges at the critical exponent $s=\delta_\Ga$, i.e.,
$\sum_{\ga\in \Ga} e^{-\delta_{\Ga} d_1(o_1, \ga o_1)}=\infty$, then $\Ga$ is said to be of
divergence type.
\begin{theorem}[\cite{Sullivan1979density}, see also {\cite[Corollary 1.8]{Roblin2003ergodicite}}] \label{thm.uniquepsdiv}
If $\Ga$ is of divergence type, there exists a unique $\Ga$-conformal measure
of dimension $\delta_\Ga$ with respect to $o_1 \in X_1$.
 In particular, it is $\Ga$-ergodic.
\end{theorem} 
We denote by $\La_{\Ga, c}\subset \La_{\Ga}$ the set of conical limit points of $\Ga$. That is,
$\xi\in \La_{\Ga, c}$ if and only if  any geodesic ray toward $\xi$ accumulates on a compact subset of $\Ga\ba X_1$.
Let $$\La_{\Ga, M}\subset \La_{\Ga, c}$$ denote the set of Myrberg limit points for $\Ga$, that is,
 $\xi\in \La_{\Ga, M}$  if and only if any geodesic ray toward $\xi$ is dense
in the union of all geodesics connecting limit points, modulo
$\Ga$. Equivalently,
   $\xi\in \La_{\Ga, M} $ if and only if for any $\eta\ne  \eta' \in  \La$, there exists a sequence $\ga_{\ell} \in \Ga$ such that $\ga_{\ell} \xi \to \eta$ and $\ga_{\ell} o_1 \to \eta'$ as $\ell \to \infty$.

The Hopf-Tsuji-Sullivan dichotomy \cite{Sullivan1979density} (see also \cite[Theorem 1.7]{Roblin2003ergodicite}) states: 

\begin{theorem} \label{thm.divergence}
 The following are equivalent to each other for any $\Ga$-conformal measure $\nu_\Ga$ of dimension $\delta_\Ga$:
\begin{enumerate}
\item $\Ga$ is of divergence type;
\item $\nu_{\Ga} (\La_{\Ga, c})=1$;
\item $\nu_{\Ga} (\La_{\Ga, M})=1$.
\end{enumerate}
\end{theorem}

A non-elementary discrete subgroup $\Ga<G_1$ is called {\it geometrically finite} (resp. convex cocompact) if
the unit neighborhood of the convex core of $\Ga\ba X_1$ has finite volume (resp. compact).
As remarked before, geometrically finite groups are of divergence type (\cite[Proposition 2]{Sullivan1984entropy}, \cite[Proposition 3.7]{Corlette1999limit}). 
Moreover, if $\Ga$ is a convex cocompact subgroup of $\so(n,1)=\Isom^+(\bH^n_{\br})$, then $\nu_\Ga$ is the $\delta_\Ga$-dimensional
Hausdorff measure on $\La$ with respect to the spherical metric on $\S^{n-1}$ \cite[Theorem 8]{Sullivan1979density}.

\subsection*{Weak-Myrberg limit set of $\Ga_\rho$}
In the rest of the section, we assume that 
$$\text{$\Ga<G_1$ is Zariski dense and
of divergence type.}$$

Let $\rho:\Ga \to G_2$ be a discrete faithful Zariski dense representation where $G_2$ is a connected semisimple real algebraic group. We use the same notations as in section \ref{set}. Let $X_2$ be
the Riemannian symmetric space associated to $G_2$. Let $G=G_1\times G_2$.
We assume that there exists a $\rho$-equivariant continuous map
$f:\La \to \F_2$. As before, we denote by
$\Ga_\rho$ the associated self-joining subgroup of $G$.  We denote by $\F^{(2)}$ the subset of $\F\times \F$ consisting of all pairs in general position.
We set $$\La_{\rho}^{(2)} =\cal F^{(2)}\cap  (\La_{\rho} \times \La_{\rho}).$$
Note that if $f$ is injective in addition, then $\La_{\rho}^{(2)}=(\La_{\rho} \times \La_{\rho})-\text{Diagonal} $. However we are not assuming the injectivity of $f$.

We recall from \cite{lee2020invariant} that
the \emph{Myrberg limit set} of $\Ga_\rho$ is defined as
     $$\La_{\rho, M} := \left\{ \xi\in \La_{\rho} : \begin{matrix}
\forall (\xi_0, \eta_0)\in \La_{\rho}^{(2)},\ \exists  \text{ a sequence }  \ga_\ell \in \Ga_\rho \mbox{ s.t.}\\
\lim \ga_\ell \xi = \xi_0 \quad \mbox{and} \quad \lim \ga_\ell o = \eta_0
\end{matrix}
\right\}.$$
\begin{definition}
   We introduce the \emph{weak-Myrberg limit set}  of $\Gr$  as follows:
     $$\La_{\rho, wM} := \bigcap_{\ga_0\in \Ga_\rho:\text{loxodromic}} \La_{\rho}(\ga_0)$$
    where
$$\La_{\rho}(\ga_0)=\left\{ \xi\in \La_{\rho} : \begin{matrix} \forall \eta \in \La_{\rho} \text{ with }(y_{\ga_0}, \eta) \in \La_{\rho}^{(2)}, \ \exists   \text{ a sequence }  \ga_\ell \in \Ga_\rho  \\ \mbox{ s.t.}
\lim \ga_\ell \xi = y_{\ga_0}\quad \mbox{and} \quad
\lim \ga_\ell o = \eta
\end{matrix}
\right\}.$$
Note that the definition does not depend on the choice of the basepoint $o$.

\end{definition}

 While arbitrary pairs $(\xi_0, \eta_0) \in \La_{\rho}^{(2)}$ are considered in the definition of the Myrberg limit set $\La_{\rho, M}$, only special pairs of the form $(y_{\ga_0}, \eta) \in \La_{\rho}^{(2)}$ with $y_{\ga_0}$ the attracting fixed point of a loxodromic $\ga_0 \in \Ga_{\rho}$ are involved in defining the weak-Myrberg limit set $\La_{\rho, wM}$. Note that 
$$\La_{\rho, M} \subset \La_{\rho, wM} \subset \La_{\rho}.$$

\begin{Def} \label{contt} We say $f$ is a {\it continuous} extension of $\rho$
if for any sequence $g_{\ell} \in \Ga$,
the convergence $g_\ell o_1\to
\xi$  implies the convergence $\rho(g_\ell) o_2\to f(\xi)$
in the sense of Definition \ref{conv}. 
\end{Def}

The rest of this section is devoted to proving the following:
\begin{prop} \label{prop.limitsetlox}\label{wM} 
Suppose either that $\op{rank}G_2=1$ or that
$f$ is a continuous extension of $\rho$. Then $$(\id \times f)_* \nu_{\Ga}(\La_{\rho, wM}) = 1.$$
\end{prop}

We begin by recalling some basic hyperbolic features of $X_1$.
The Gromov product of $x, y\in X_1$ at $p \in X_1$ 
is defined as
$$\langle x, y\rangle_p=\frac{1}{2} \left( d_1(x,p)+d_1(p, y)- d_1(x,y) \right).$$
For $\xi\ne \eta\in  \F_1$, the Gromov product $\langle \xi, \eta\rangle_p$
is defined as
$$\langle \xi, \eta\rangle_p=\lim_{\ell \to \infty} \langle x_\ell, y_\ell\rangle_p$$  for any sequences $x_\ell, y_\ell\in X_1$ converging 
to $\xi, \eta$ respectively. We also set $\langle \xi, \xi \rangle_p = \infty$. 
The Gromov product $\langle \xi, \eta\rangle_p$ is roughly a
distance from $p$ to the geodesic connecting $\xi$ and $\eta$; more precisely, there exists a constant $c > 0$ depending only
on the  (Gromov) hyperbolicity constant of $X_1$ such that for any $p\in X_1$ and $\xi, \eta, \zeta \in X_1\cup {\cal F}_1$,  \be\label{gro} |d(p, [\xi, \eta]) - \langle \xi, \eta \rangle_p | < c \quad\text{and} \quad
\langle \xi, \eta \rangle_p \ge \min \{ \langle \xi, \zeta \rangle_p, \langle \zeta, \eta \rangle_p \} - c
\ee
 where $[\xi, \eta]$ denotes the unique geodesic in $X_1$ connecting $\xi$ and $\eta$.
For any $q\in [\xi, \eta]$, we have
$$\langle \xi, \eta\rangle_p =\frac{1}{2} \left(\beta_\xi (p, q)+\beta_\eta(p, q)\right).$$ 

The visual metric  on  $\F_1$ at $p$ is defined by
$$d_{p}(\xi, \eta)=e^{-\langle \xi, \eta\rangle_{p} }\;\; \text{if $\xi\ne \eta$ and }
d_{p}(\xi, \xi)=0 .$$ 
If we normalize the metric so that $X_1$
has the sectional curvature at most $-1$, then $d_p$
is indeed a {\it metric}; this was proved by Bourdon \cite[Section 1.1]{Bourdon1995structure}.

\subsection*{Proof of Proposition \ref{wM}}
If $\Ga$ is of divergence type, then $\nu_{\Ga}(\La_{\Ga, M}) = 1$ by Theorem \ref{thm.divergence}. Therefore Proposition \ref{wM} 
 is an immediate consequence of the following:
\begin{lemma} \label{lem.loxpush}
We have $$ (\id \times f)\left( \La_{\Ga, M}\right) \subset \La_{\rho, wM}.$$
\end{lemma}

\begin{proof}  Let $\xi_1\in \La_{\Ga, M}$ be an arbitrary element and
let $\xi=(\xi_1, f(\xi_1)) $.
Letting 
$\ga_0=(g_0, \rho(g_0))\in \Ga_\rho$ be an arbitrary loxodromic element,
we need to show that
\be\label{rrr} \xi\in \La_{\rho}(\ga_0) .\ee 
Noting that $y_{\ga_0} = (y_{g_0}, f(y_{g_0})) \in \La_{\rho}$,
let $\eta = (\eta_1, f(\eta_1)) \in \La_{\rho}$ be any element which is in a general position with $y_{\ga_0}$. As $\xi_1 \in \La_{\Ga, M}$, there exists a sequence $g_{\ell} \in \Ga$ such that 
\be \label{ex}
 g_{\ell} \xi_1 \to y_{g_0} \quad \mbox{and} \quad g_{\ell} o_1 \to \eta_1 \quad \mbox{as } \ell \to \infty.
\ee
Take any $\xi_{-1} \in \La - f^{-1}(f(\xi_1))$; this is possible since $\rho(\Ga)$ is Zariski dense and hence $f(\La)$ is not singleton. 
Since for all $\ell \ge 1$,
$$\langle g_\ell \xi_1, g_\ell \xi_{-1} \rangle_{g_\ell o_1} = \langle \xi_1, \xi_{-1} \rangle_{o_1} < \infty,$$ 
\eqref{ex}  and \eqref{gro} imply that 
\be\label{gm} \lim_{\ell \to \infty} g_{\ell} \xi_{-1} = \lim_{\ell \to \infty} g_{\ell} o_1 = \eta_1 .\ee

Since $\lim_{\ell \to \infty} g_\ell o_1 = \eta_1 \neq y_{g_0}$ and $\lim_{\ell \to \infty} g_\ell \xi_1 = y_{g_0}$, we have
$$
\langle \xi_1, g_\ell^{-1}y_{g_0}\rangle_{o_1} = \langle g_\ell \xi_1, y_{g_0}\rangle_{g_\ell o_1} \to \infty \quad \mbox{as } \ell \to \infty.
$$
It follows that
\be \label{eqn.gromovobs}
\lim_{\ell \to \infty} g_\ell^{-1} y_{g_0} = \xi_1.
\ee

As $f$ is $\rho$-equivariant and continuous,
we get from \eqref{ex}, \eqref{gm} and Lemma \ref{ll} that
as $\ell \to \infty$,  $$
\rho(g_\ell)f(\xi_{1}) = f(g_\ell \xi_{1}) \to f(y_{g_0}) = y_{\rho(g_0)} \text{ and } $$
$$
\rho(g_{\ell})f(\xi_{-1}) = f(g_{\ell} \xi_{-1}) \to f(\eta_1).
$$

Hence \eqref{rrr} follows once we verify that 
\be\label{ver} \lim_{\ell \to \infty} \rho(g_\ell)o_2 = f(\eta_1) .\ee 
If $f$ is a 
continuous extension of $\rho$, this is automatic from \eqref{ex}. Hence we now assume that $\op{rank}G_2=1$, and so $X_2$ is a rank one symmetric space, in the rest of the proof. We use the same notation $\langle \cdot, \cdot \rangle_{\cdot}$ for the Gromov product in $X_2 \cup \F_2$.
By the hyperbolicity of $X_2$, we have a constant $c > 0$ depending only on $X_2$  as in \eqref{gro} so that for all $\ell\ge 1$, $$
\begin{aligned}
\langle f(\xi_1), f(\xi_{-1})\rangle_{o_2} & = \langle f(g_\ell \xi_1), f(g_\ell \xi_{-1})\rangle_{\rho(g_\ell)o_2} \\
& \ge \min \{ \langle f(g_\ell \xi_1), y_{\rho(g_0)}\rangle_{\rho(g_\ell)o_2}, \langle y_{\rho(g_0)}, f(g_\ell \xi_{-1})\rangle_{\rho(g_\ell) o_2} \} - c\\
& = \min \{ \langle f(\xi_1), \rho(g_\ell)^{-1}y_{\rho(g_0)}\rangle_{o_2}, \langle y_{\rho(g_0)}, f(g_\ell \xi_{-1})\rangle_{\rho(g_\ell) o_2} \} - c.
\end{aligned}$$
On the other hand, by \eqref{eqn.gromovobs} and the continuity of $f$, we have,
as $\ell \to \infty$,
$$\rho(g_{\ell})^{-1}y_{\rho(g_0)} = f(g_{\ell}^{-1} y_{g_0}) \to f(\xi_1)$$  which implies $\langle f(\xi_1), \rho(g_{\ell})^{-1}y_{\rho(g_0)}\rangle_{o_2} \to \infty$ as $\ell \to \infty$. Hence for all large enough $\ell \ge 1$, we have 
$$\langle f(\xi_1), f(\xi_{-1}) \rangle_{o_2} \ge \langle y_{\rho(g_0)}, f(g_{\ell}\xi_{-1}) \rangle _{\rho(g_{\ell})o_2} - c.$$
Again, it follows from the hyperbolicity of $X_2$ that 
\be \label{eqn.gromov2}
\langle f(\xi_1), f(\xi_{-1}) \rangle_{o_2} \ge \min\{ \langle y_{\rho(g_0)},  f(\eta_1) \rangle_{\rho(g_\ell) o_2}, \langle f(\eta_1), f(g_{\ell} \xi_{-1}) \rangle_{\rho(g_{\ell}) o_2}\} - 2c.
\ee

Now suppose to the contrary that \eqref{ver} does not hold.
From the choice of $\xi_{-1}$, we have $f(\xi_1) \neq f(\xi_{-1})$ and hence $$ \langle f(g_\ell \xi_1), f(g_\ell \xi_{-1})\rangle_{\rho(g_\ell)o_2} = \langle f(\xi_1), f(\xi_{-1})\rangle_{o_2}<\infty . $$ 
Since $f(g_\ell \xi_{1}) \to y_{\rho(g_0)}$ and $f(g_{\ell} \xi_{-1}) \to f(\eta_1)$ as $\ell \to \infty$,  by passing to a subsequence,
we may assume that  $\rho(g_\ell)o_2 $ converges to either 
$y_{\rho(g_0)}$ or $f(\eta_1)$. Since we are assuming that \eqref{ver} does not hold, we must have
$$\lim_{\ell \to \infty} \rho(g_\ell)o_2 = y_{\rho(g_0)}.$$

Since $\lim_{\ell \to \infty} f(g_{\ell} \xi_{-1}) = f(\eta_1) \neq y_{\rho(g_0)}$, it implies that $$\langle f(\eta_1), f(g_{\ell}\xi_{-1})\rangle_{\rho(g_{\ell}) o_2} \to \infty \quad \mbox{as } \ell \to \infty.$$
Then, for all large enough $\ell \ge 1$, we have from \eqref{eqn.gromov2} that $$  \langle f(\xi_1), f(\xi_{-1})\rangle_{o_2} \ge \langle y_{\rho(g_0)}, f(\eta_1) \rangle_{\rho(g_{\ell})o_2} -2c.$$

 It follows from \eqref{gro} that the distance between
$\rho(g_{\ell}) o_2$ and the geodesic $[y_{\rho(g_0)}, f(\eta_1)]$ connecting $y_{\rho(g_0)}$ and $f(\eta_1)$ is uniformly bounded. Hence for some $B > 0$, $\rho(g_{\ell})o_2$ converges to $y_{\rho(g_0)}$ within the $B$-neighborhood of the translation axis $\cal L_{\rho(g_0)}$ of $\rho(g_0)$. Since $\rho(g_0)$ acts cocompactly on any fixed neighborhood of $\cal L_{\rho(g_0)}$,
there exists a sequence  $n_\ell \to \infty$ and a compact subset $\C \subset X_2$ such that $$\rho(g_0)^{-n_\ell} \rho(g_\ell) o_2 \in \C\cap \rho(\Ga)o_2.$$ Since $\C  \cap \rho(\Ga)o_2$ is a finite subset,  by passing to a subsequence,
we may assume that
$\rho(g_0)^{-n_\ell}\rho(g_\ell)$ is a constant sequence in $\rho(\Ga)$, say, $\rho(g)$, for some $g \in \Ga$. 
As $\rho$ is faithful, it follows that
$$g_{\ell} = g_0^{n_\ell} g$$ for all $\ell \ge 1$, and hence $g_\ell o_1 $ converges to
$ y_{g_0}$ as $\ell \to \infty$; this contradicts
the second condition in \eqref{ex} as $\eta_1\ne y_{g_0}$.
Therefore we have proved \eqref{ver}, completing the proof.
\end{proof}

\section{Essential subgroups for graph-conformal measures} \label{sec.ess}
Let $G_1$ be a connected simple real algebraic group of rank one and $G_2$ be a connected semisimple real algebraic group.  We use the notations set up in section \ref{set} such as $G=G_1\times G_2$ and $\F=\F_1\times \F_2$.

Let $\Ga<G_1$ be a Zariski dense discrete subgroup of divergence type
 and $\rho : \Ga \to G_2$ be a discrete faithful Zariski dense representation. 
 Let $f : \La \to \F_2$ be a $\rho$-equivariant continuous map where $\La\subset \F_1$ is the limit set of $\Ga$.
Let $\nu_\Ga$ be the unique $\Ga$-conformal measure of dimension $\delta_{\Ga}$ on $\La$ and
set $$\nug=(\id \times f)_*\nu_\Ga.$$
In this case, $\fa_1=\br$ and $\nug$ is a $(\Ga_\rho, \sigma_1)$-conformal measure where 
\be \label{eqn.defsigma}
\sigma_1(u_1, u_2)=\delta_\Ga u_1 .
\ee

The main goal of this section is to prove: 
\begin{theorem} \label{thm.ess}
Let $\Ga < G_1$ be of divergence type. Suppose either that $\op{rank}G_2=1$ or that
$f$ is a continuous extension of $\rho$.
If $\Ga_\rho$ is Zariski dense, then 
$$\ess_{\nug} (\Ga_{\rho}) = \fa.$$

\end{theorem}

\begin{Rmk} 
\begin{enumerate}
\item In fact, our proof shows a slightly stronger statement that
for any non-trivial
normal subgroup $\Ga ' < \Ga$,
$$\ess_{\nug} (\Ga'_{\rho}) = \fa.$$

\item 
  If $\Gr$ is an Anosov subgroup with respect to $P$, Theorem \ref{thm.ess} is a special case of  \cite[Proposition 10.2]{lee2020invariant}. The proof there uses  the Anosov property in a crucial way.  
Our proof of Theorem \ref{thm.ess} is an adaptation of this proof to graph-conformal measures of a self-joining subgroup, which is not necessarily Anosov. 
  \end{enumerate}
\end{Rmk}

\subsection*{Covering lemma}

Let $p = (p_1, p_2) \in X$, $\xi = (\xi_1, \xi_2), \eta = (\eta_1, \eta_2) \in \La_{\rho}$.
As  $\La_\rho$ is the graph of $f$, we have $ \xi \ne  \eta  \in \La_\rho$ implies that
$\xi_1\ne \eta_1$. 
We define a metric-like function $d_p$
on $\La_\rho$ as follows: for any $ \xi\ne  \eta \in \La_\rho$, set
\be \label{eqn.defdp}
d_p(\xi, \eta) := e^{-\delta_\Ga \cdot  \langle \xi_1, \eta_1 \rangle_{p_1} } 
:= d_{p_1}(\xi_1, \eta_1)^{\delta_\Ga},
\ee
and $d_p(\xi, \eta)=0$ if $\xi=\eta$. 
 Let $n_0 > 0$ be a constant so that the normalized metric space $(X_1, n_0 d_1)$ has  sectional curvature at most $-1$. Then, as remarked before, $d_{p_1}^{n_0} = d_p^{n_0/\delta_{\Ga}}$ is a genuine metric by \cite[Section 1.1]{Bourdon1995structure}.
We fix $N_0 \ge 1$  such that  for
all $a, b\ge 0$,
$$\left(a^{n_0 \over \delta_{\Ga}} + b^{n_0 \over \delta_{\Ga}}\right)^{\delta_{\Ga} \over n_0} \le N_0 (a+b) .
$$
Then
we have the following pseudo-triangle inequality:
for all $\xi, \eta, \zeta\in \La_\rho$,
\begin{align}\label{no} d_p(\xi, \eta) & = \left(d_{p}(\xi, \eta)^{n_0 \over \delta_{\Ga}} \right)^{\delta_\Ga \over n_0} \nonumber \\
&\le  \left(d_{p}(\xi, \zeta)^{n_0 \over \delta_{\Ga}} + d_{p}(\zeta, \eta)^{n_0 \over \delta_{\Ga}} \right)^{\delta_\Ga \over n_0} \nonumber \\
& \le N_0 (d_{p}(\xi, \zeta) + d_{p}(\zeta, \eta)).
\end{align}

 For $\xi \in \La_{\rho}$ and $r > 0$, set
 $$B_p(\xi, r) := \{ \eta \in \La_{\rho} : d_p(\xi, \eta) < r \}.$$
It is standard to deduce the following from the triangle inequality \eqref{no} (cf. \cite[Lemma 6.12]{lee2020invariant}):
\begin{lemma} [Covering lemma] \label{lem.covering}
For any finite collection of open $d_p$-balls $B_p(\xi_1, r_1), \cdots, B_p(\xi_n, r_n)$ for $\xi_j \in \La_{\rho}$ and $r_j > 0$, there is a subcollection of disjoint balls $B_p(\xi_{i_1}, r_{i_1}), \cdots, B_p(\xi_{i_{\ell}}, r_{i_{\ell}})$ such that $$\bigcup_{j = 1}^n B_p(\xi_j, r_j) \subset \bigcup_{j = 1}^{\ell} B_p(\xi_{i_j} ,3N_0 r_{i_j}).$$ 
where $N_0\ge 1$ is as in \eqref{no}.
\end{lemma}

\subsection*{Shadow lemma} 
For $g\in G$, its visual images are defined by
$$g^+=gP\in \F\quad\text{and}\quad  g^-=gw_0 P\in \F$$
where  $w_0\in G$ is an element of the normalizer of $A$
such that $w_0Pw_0^{-1}\cap P=AM$. 

Let $x, y \in X$ and $r > 0$.
The shadow of the ball $$B(y, r)=\{z\in X: d(z,y)<r\}$$ viewed from $x$ is defined as follows: $$O_r(x, y) := \{ gk^+ \in \F : k \in K,
gk (\text{int}A^+ )o\cap B(y, r)\ne \emptyset  \}$$ where $g \in G$ satisfies $x = g o$. The shadow of $B(y,r)$
viewed from $\xi \in \F$ can also be defined: $$O_r(\xi, y) := \{ h^+ \in \F : h^- = \xi, ho\in B(y,r) \}.$$

 The following is an analogue of Sullivan's shadow lemma:
\begin{lemma}[{\cite[Lemma 5.7]{lee2020invariant}}] \label{lem.shadow}
There exists $\kappa > 0$ such that for any $x, y  \in X$ and $r > 0$, we have $$\sup_{\xi \in O_r(x, y)} \| \beta_{\xi}(x, y) - \underline{a}(x, y) \| \le \kappa r$$
where $\underline{a}(x, y)$ is as defined in Definition \ref{under}.
\end{lemma}

\subsection*{Jordan projections of self-joinings}
We will need the following lemma which we deduce from Theorem \ref{dense0}:
\begin{lem}\label{finite} Suppose that $\Ga_\rho$ is Zariski dense in $G$.
For any $Q>0$, the subset
$$\{ \lambda(\ga)\in \fa :
\ga \in \Ga_\rho, \sigma_1(\lambda(\ga)) \ge  Q\}$$ generates a dense subgroup of $\fa$.
\end{lem}
\begin{proof} 
Note that for a given $Q>0$,
$\{ \lambda(\ga)\in \fa :
\ga \in \Ga_\rho, \sigma_1(\lambda(g)) \le  Q\}$ is not  a finite subset in general. Hence this is not a direct consequence of Theorem \ref{dense0}.
On the other hand,
as $\Ga_\rho$ is Zariski dense,
we can find a Zariski dense Schottky subgroup $\G'<\G$ such that $\Gr'$ is Zariski dense in $G$ (see for example \cite[Lemma 7.3]{edwards2020anosov}).
Since $\G'$ is convex cocompact, there
are only finitely many closed geodesics in $\Ga'\ba X_1$ of length bounded by a fixed number. Since the set of closed geodesics in  $\Ga'\ba X_1$  is in one to one correspondence with  the set $[\G']$ of conjugacy classes of loxodromic elements while the Jordan projection of a loxodromic element is the length of the corresponding closed geodesic, we have
  $$\#
\{[g']\in [\G']:
 \lambda(g' ) < \delta_{\Ga}^{-1} Q\}<\infty  .$$
Therefore
$\# \{ [\ga']\in [\G'_\rho]: \sigma_1(\lambda(\ga'))<Q \} < \infty$.
Hence  Theorem \ref{dense0} implies that
 $\{\lambda(\ga')\in \fa :
\ga \in \G'_{\rho}, \sigma_1(\lambda(\ga')) \ge  Q\}$ generates a dense subgroup of $\fa$. This implies the claim.
\end{proof}

\subsection*{Main proposition}
Recall the constant $N_0$ from Lemma \ref{lem.covering}.
Since $\ess_{\nug}(\Gr)$ is a closed subgroup of $\fa$,
Theorem \ref{thm.ess} follows from the following proposition
by Lemma \ref{finite}:

\begin{prop} \label{prop.ess} \label{prop.jordaness}  
Let $\Ga, \rho, f$ be as in Theorem \ref{thm.ess}.
Let $\ga_0 \in \Gr$ be any loxodromic element such that $\sigma_1(\la(\ga_0)) > 1 + \log 3 N_0$.
For any $\varepsilon > 0$ and Borel subset $B \subset \F$ with $\nug(B) > 0$, there exists $\ga \in \Gr$ such that $$
B \cap \ga \ga_0 \ga^{-1} B \cap \left\{ \xi \in \Lambda_{\rho} : \begin{matrix}
\| \beta_{\xi}(o, \ga \ga_0 \ga^{-1} o) - \la(\ga_0) \| < \varepsilon
\end{matrix} \right\}$$
has a positive $\nug$-measure. In particular,
$$\lambda(\ga_0)\in \ess_{\nug}(\Ga_{\rho}).$$
\end{prop}

The rest of this section is devoted to the proof of Proposition \ref{prop.ess}.
In the rest of this section, we fix a loxodromic element  $\ga_0 \in \Gr$ 
such that $\sigma_1(\la(\ga_0)) > 1 + \log 3 N_0$.
Set $$\xi_0:=y_{\ga_0}$$
and
let $\eta\in \La_\rho$ be any element which is in general position with $ \xi_0$. Since $(\xi_0, \eta)\in \La_\rho^{(2)}$,
we can choose 
$p = go \in X$ where $g \in G$ such that $g^+ = \xi_0$ and $g^- = \eta$.

We also fix $$0 < \varepsilon < \min(1/2, \delta_{\Ga}^{-1}).$$

\subsection*{Covering $\La_{\rho, wM}$ by a certain collection of $d_p$-balls}
 We proceed in the same way as in \cite[Section 10]{lee2020invariant}.  
For each $\ga \in \Gr$, let $r_p(\ga) > 0$ be the supremum  of $r \ge 0$ such that $$\max_{\xi \in B_p(\ga \xi_0, 3N_0 r)} \| \beta_{\xi}(p, \ga \ga_0^{\pm 1} \ga^{-1} p ) \mp \la(\ga_0) \| < \varepsilon$$
where $N_0$ is as in \eqref{no}.
Consider the following collection of $d_p$-balls for each $R > 0$: $$\B_{R}(\ga_0, \varepsilon) = \{ B_p(\ga \xi_0, r) : \ga \in \Gr, 0 < r < \min(R, r_p(\ga))\} .$$

\begin{lemma} \label{lem.loxcovering} There exists $s(\ga_0)>0$ such that for any $R > 0$, the following holds:
if $\xi \in \La_{\rho}$ and $\ga_{\ell} \in \Gr$ is a sequence such that $\ga_\ell^{-1} p \to \eta$ and $\ga_\ell^{-1}\xi \to y_{\ga_0} $, then for any $0 < r \le \min(s(\ga_0), R)$, there exists $\ell_0 = \ell_0(r) > 0$ such that for all $\ell \ge \ell_0$, $$D(\ga_\ell \xi_0, r) \in \B_{R}(\ga_0, \varepsilon) \quad \mbox{and} \quad \xi \in D(\ga_\ell \xi_0, r)$$
where
$$D(\ga_\ell \xi_0, r) := B_p \left(\ga_\ell \xi_0, {1 \over 3 N_0} e^{-\sigma_1(2 \underline{a}(\ga_\ell^{-1} p, p))} r\right).$$

In particular, for any $R>0$, $$\La_{\rho, wM} \subset \bigcup_{D \in \B_R(\ga_0, \varepsilon)} D.$$
\end{lemma}

\begin{proof}
The second assertion is an immediate consequence of the first by the definition of $\La_{\rho, wM}$.
The first claim of the lemma is proved as \cite[Lemma 10.12]{lee2020invariant} for the case when $\Gr$ is Anosov. 
Since $|\beta_{\xi_1}(p_1,q_1)|\le d_1(p_1, q_1) $ for any $p_1, q_1\in X_1$ and $\xi_1\in {\cal F}_1$, we have that
for all $\ga \in \Gr$ and $\xi \in \La_{\rho}$,
$$ - \sigma_1(\underline{a}(p, \ga p)) \le \sigma_1(\beta_{\xi}(p, \ga p)) \le \sigma_1 (\underline{a}(p, \ga p)) .$$
Substituting \cite[Theorem 5.3]{lee2020invariant} with this special property of the linear form $\sigma_1$ and using  Lemma \ref{lem.shadow},
the proof of \cite[Lemma 10.12]{lee2020invariant} can be repeated verbatim. We only explain how to define $s(\ga_0)$:
by the choice of $p\in X$, we have
$\xi_0 \in O_{\varepsilon \over 8 \kappa}(\eta, p)$ where $\kappa > 0$ is the constant in Lemma \ref{lem.shadow}.
Therefore we can choose $s = s(\ga_0) >0$ so that $$B_p(\xi_0, e^{\sigma_1(2\la(\ga_0)) + {1 \over 2} \delta_{\Ga} \varepsilon} s) \subset O_{\varepsilon \over 8 \kappa} (\eta, p);$$
$$\sup_{x \in B_p(\xi_0, s)} \| \beta_x (p, \ga_0^{\pm 1} p) \mp \la (\ga_0) \| < {\varepsilon \over 4}.$$
\end{proof}

\subsection*{Approximating Borel subsets by $d_p$-balls in $\cal B_R(\ga_0, \e)$}
It is more convenient to use
the following conformal measure $\nu_p$ (with respect to the basepoint $p$): \be \label{eqn.nup}
d \nu_{p}(\xi) = e^{\sigma_1(\beta_{\xi}(o, p))} d\nug(\xi)
\ee

We now use $$\nu_p(\La_\rho-\Lambda_{\rho, wM})=0$$ (Proposition \ref{prop.limitsetlox})  to show the following:
\begin{prop} \label{prop.submerge}
    Let $B \subset \F$ be a Borel subset with $\nu_p(B) > 0$. Then  for $\nu_p$-a.e. $\xi\in B$,  $$\lim_{R\to 0} \sup_{\xi\in D \in \B_{R}(\ga_0, \varepsilon)} {\nu_p(B \cap D) \over \nu_p(D)} = 1.$$
\end{prop}

\begin{proof}
Associated to a measurable function $h : \F \to \R$, we define $$h^*(\xi) = \lim_{R \to 0} \sup_{D \in \B_R(\ga_0, \varepsilon), \xi \in D} {1 \over \nu_{p}(D)} \int_{D} h d\nu_{p}.$$
By Lemma \ref{lem.loxcovering}, $h^*$ is well-defined on $\La_{\rho, wM}$. Hence by Proposition \ref{prop.limitsetlox}, $h^*$ is well-defined $\nu_p$-almost everywhere.

The desired statement is obtained by showing that $h^* = h$ and then by taking $h = \mathbf{1}_B$.  Again this is proved in \cite[Proposition 10.17]{lee2020invariant}, and the key ingredient is the covering lemma for the Anosov setting as stated in \cite[Lemma 6.12]{lee2020invariant}. Substituting this with our Lemma \ref{lem.covering}, together with the choice $\varepsilon < \delta_{\Ga}^{-1}$ and $\sigma_1(\la(\ga_0)) > 1 + \log 3 N_0$, we can 
repeat
the proof of \cite[Proposition 10.17]{lee2020invariant} verbatim with $\psi=\sigma_1$. 
\end{proof}

\begin{remark}
In \cite{lee2020invariant}, the Anosov property was used to ensure that the Myrberg limit set 
$\La_{\rho, M}$ has full $\nu_p$-measure, and hence $h^*$ is well-defined $\nu_p$-almost everywhere. In our setting, we use that the weak-Myrberg limit set $\La_{\rho, wM}$ has full $\nu_p$-measure (Proposition \ref{wM}). This was sufficient to prove Proposition \ref{prop.submerge}, without Anosov property.
\end{remark}

\subsection*{Proof of Proposition \ref{prop.ess}}
Anosov version of this proposition is \cite[Proposition 10.7]{lee2020invariant};
the key ingredients of its proof are \cite[Lemma 10.12, Propostion 10.17]{lee2020invariant}. Substituting
these respectively by  Lemma \ref{lem.loxcovering} and Proposition \ref{prop.submerge},
the proof works in the same way as \cite{lee2020invariant}. We give a brief sketch. Since $\ess_{\nug}(\Gr) = \ess_{\nu_p}(\Gr)$, it suffices to show that 
$$\la(\ga_0) \in \ess_{\nu_p}(\Gr).$$
 Let $B \subset \F$ be any Borel subset
 with $\nu_{p}(B) > 0$ and $\e>0$ be any sufficiently small number. 
By Proposition \ref{prop.submerge},  there exist $\ga\in \Ga_\rho$ and $0<r<r_p(\ga)$
such that \be \label{eqn.choiceD}
\nu_p(B \cap D) > ( 1 + e^{ - \sigma_1(\la(\ga_0)) - \delta_{\Ga} \varepsilon})^{-1}  \nu_p(D)
\ee
where $D = B_p(\ga \xi_0, r)$.

Since $r < r_p(\ga)$, we have by the definition of $D = B_p(\ga \xi_0, r)$ and $r_p(\ga)$,
\be \label{eqn.D}
D \subset \{ \xi \in \La_{\rho} : \| \beta_{\xi}(p, \ga \ga_0^{\pm 1} \ga^{-1} p) \mp \la(\ga_0) \| < \varepsilon \}.
\ee
This implies that
\be\label{ff} (B \cap D) \cap \ga \ga_0 \ga^{-1} (B \cap D) \subset B \cap \ga \ga_0 \ga^{-1} B \cap \{ \xi : \| \beta_{\xi}(p, \ga \ga_0 \ga^{-1} p) - \la(\ga_0) \| < \varepsilon \}.\ee 

By the conformality of $\nu_p$ and \eqref{eqn.D}, we have 
$$
\nu_p(B \cap D) + \nu_p(\ga \ga_0 \ga^{-1}(B \cap D)) \ge (1 + e^{-\sigma_1(\la(\ga_0)) - \delta_{\Ga} \varepsilon}) \nu_p(B \cap D).
$$
It follows from \eqref{eqn.choiceD} that 
\be \label{eqn.bigmeasure}
\nu_p(B \cap D) + \nu_p(\ga \ga_0 \ga^{-1}(B \cap D)) > \nu_p(D).
\ee

Using that $\varepsilon < \delta_{\Ga}^{-1}$ and $\sigma(\la(\ga_0)) > 1 + \log 3 N_0$, we can check that
$$\ga \ga_0 \ga^{-1}D \subset D. $$ 
Hence the left hand side of \eqref{eqn.bigmeasure} is at most $\nu_p(D)$.
It follows that
$$\nu_p( (B \cap D) \cap \ga \ga_0 \ga^{-1} (B \cap D))>0.  $$
By \eqref{ff}, this implies that
$$
B \cap \ga \ga_0 \ga^{-1} B \cap \{ \xi : \| \beta_{\xi}(p, \ga \ga_0 \ga^{-1} p) - \la(\ga_0) \| < \varepsilon \}$$
has positive $\nu_p$-measure. Since $B$ and $\varepsilon>0$ are arbitrary, this proves that $\la(\ga_0)\in \ess_{\nu_p}(\Ga_\rho)$, finishing the proof.

\section{Dichotomy theorems for Zariski density of $\Ga_\rho$} \label{sec.proof}

We are finally ready to prove our main theorems. Let $G_1$ be a connected simple real algebraic group of rank one, and $\Ga < G_1$ a Zariski dense discrete subgroup of divergence type.
Let $G_2$ be a connected semisimple real algebraic group.
Let $\rho:\Ga\to G_2$ be a discrete faithful Zariski dense
representation admitting the continuous equivariant map $f:\La\to \F_2$. We use  the notations from section \ref{set} and section \ref{sec.ess}.

We suppose in the entire section that
$$\text{either $\op{rank} G_2 = 1$ or that $f$ is a continuous extension of $\rho$.}$$

Recall that $\nug=(\id\times f)_*\nu_\Ga$
is the unique $(\Ga_\rho, \sigma_1)$-conformal measure on $\La_\rho$, where $\sigma_1$ is the linear form on $\fa_1 \oplus \fa_2 = \R \oplus \fa_2$ defined by $$\sigma_1(u_1, u_2) = \delta_{\Ga} u_1.$$
We sometimes write $\nug = \nu_{\sigma_1}$. We establish the following dichotomy of which Theorem \ref{km} is a special case:
\begin{theorem}\label{mp} \label{mplast} 
 In each of the two complementary cases,
 the claims (1)-(4)  are equivalent to each other.
 We assume $G_2$ is simple only for the implications $(4) \Rightarrow (1)$ in the first case and $(1) \Rightarrow (2), (3), (4)$ in the second case.

 The first case is as follows:
 \begin{enumerate}
\item $\Ga_\rho$ is Zariski dense in $G$;
\item $\mathsf E_{\nug}(\Ga_\rho)=\fa$;
\item $m_{\nug}^{\BR}$ is $NM$-ergodic;
\item For any $(\Ga_\rho,\psi)$-conformal measure $\nu$
on $\La_\rho$ for $\psi\ne \sigma_1$, we have $[\nug]\ne [\nu]$.
\end{enumerate}

The second case is as follows (in this case $\op{rank} G_2=1$ as a consequence):
\begin{enumerate}
\item $\Ga_\rho$ is not Zariski dense in $G$;
\item $\fa = \R^2$ and $\mathsf E_{\nug}(\Ga_\rho)= \{ t \cdot (\delta_{\rho(\Ga)}, \delta_{\Ga}) \in \R^2 : t \in \R \}$;
\item $m_{\nug}^{\BR}$ is not $NM$-ergodic;
\item For any $(\Ga_\rho, \psi)$-conformal measure $\nu$
on $\La_\rho$ for a tangent linear form $\psi$, we have $[\nug]= [\nu]$. 
\end{enumerate}
\end{theorem}

 \begin{proof}
 We prove the equivalences in each of two cases.  
 
 \medskip
{\noindent \bf First case:}
By Proposition \ref{prop.lo}, the equivalence $(2) \Leftrightarrow (3)$ follows once we note that $\nu_{\op{graph}}$ is a $\Gr$-ergodic conformal measure. Since $\Ga$ is of divergence type, $\nu_{\Ga}$ is the unique $\Ga$-conformal measure by Theorem \ref{thm.divergence}. Hence, by Corollary \ref{cor.uniqueconformal}, we indeed have that $\nu_{\op{graph}}$ is a $\Ga_{\rho}$-ergodic conformal measure. This implies $(2) \Leftrightarrow (3)$.

The implication $(1)\Rightarrow (2)$ was proved in Theorem \ref{thm.ess}. To show the implication $(2) \Rightarrow (4)$, suppose that there exists a $(\Ga_{\rho}, \psi)$-conformal measure $\nu$ on $\La_{\rho}$ with $\psi \neq \sigma_1$ such that $[\nu_{\op{graph}}] = [\nu]$. In particular, $\nu \ll \nu_{\op{graph}}$. Since $\nu_{\op{graph}}$ is $(\Ga_{\rho}, \sigma_1)$-conformal, it follows from Proposition \ref{prop.loabscont} that $\sigma_1$ and $\psi$ coincide on $\ess_{\nu_{\op{graph}}}(\Ga_{\rho})$. Since $\sigma_1 \neq \psi$, this implies that $\ess_{\nu_{\op{graph}}}(\Ga_{\rho}) \neq \fa$. This shows the contrapositive of $(2) \Rightarrow (4)$, as desired.
In order to prove $(4) \Rightarrow (1)$,
 we  assume that (4) holds and that $\Gr$ is not Zariski dense in $G$. 
  By Lemma \ref{Zdense},
 $\rho : \Ga \to G_2$ extends to a Lie group isomorphism $G_1 \to G_2$, which we also denote by $\rho$ by abuse of notation. In particular, $\op{rank} G_2 = 1$.

 For each $i=1,2$, let $\alpha_i \in \fa_i^*$ be the simple root of $(\fg_i, \fa_i^+)$ and define $\sigma_i\in \fa^*$ by
 \be \label{eqn.defsigma2}
\sigma_1(u_1, u_2)  = \delta_{\Ga} \alpha_1(u_1)\text{ and }  \sigma_2(u_1, u_2)  = \delta_{\rho(\Ga)} \alpha_2(u_2).
 \ee

Since there is a unique left $G_i$-invariant Riemannian metric on $X_i$
up to a constant multiple, we may assume that the metric $d_i$ is the one induced from the Killing form of $\frak g_i$ for each $i=1,2$. 
Note that for each $\xi \in \F_1$ (resp. $ \F_2$) and $x, y \in X_1$ (resp. $ X_2$), the product $\delta_{\Ga} \beta_{\xi}(x, y)$ (resp. $\delta_{\rho(\Ga)} \beta_{\xi}(x, y)$) does not change after scaling the metric $d_1$ (resp. $d_2$) by the definition of the critical exponent. 

 Since the differential of $\rho$ must send the Killing form of $\fg_1$ to that of $\fg_2$,
  $\rho$ induces an isometry $X_1=G_1/K_1\to X_2=G_2/\rho(K_1)$, which we again denote by $\rho$, as well
  as the equivariant diffeomorphism $F:\F_1= G_1/P_1\to \F_2= G_2/\rho(P_1)$, so that $f=F|_\La$.
 It follows that $$\delta_\Ga =\delta_{\rho(\Ga)}$$
 and for all $\xi\in \F_1 $ and $x,y\in X_1$,
 $$\beta_\xi(x,y)= \beta_{F(\xi)}(\rho(x), \rho(y)).$$ For simplicity, we set $\delta :=\delta_\Ga=\delta_{\rho(\Ga)}$ below.

We claim that $\nu_{\sigma_1} = (\id \times f)_* \nu_{\Ga}$ is $(\Ga_\rho,\sigma_2)$-conformal.  First note that $f_*\nu_{\Ga}$ is a $(\rho(\Ga), \delta)$-conformal measure on $\La_{\rho(\Ga)}$ with respect to $o_2:=\rho(o_1)$:
for any $g \in \Ga$ and $\xi \in \F_1$, $$
\begin{aligned}
{d \rho(g)_*f_*\nu_{\Ga} \over d f_*\nu_{\Ga}}(f(\xi)) & = {d f_*g_*\nu_{\Ga} \over df_* \nu_{\Ga}}(f(\xi)) = {d g_* \nu_{\Ga} \over d\nu_{\Ga}}(\xi) \\
& = e^{\delta\beta_{\xi}(o_1, go_1)} = e^{\delta \beta_{f(\xi)}(o_2, \rho(g)o_2)}.
\end{aligned}$$
 Therefore, by Lemma \ref{s2},
 the pushforward $(f^{-1} \times \id)_* f_* \nu_{\Ga}$ is a $(\Gr, \sigma_2)$-conformal measure on $\La_{\rho}$.
Since  
$$ \nu_{\sigma_1} = (\id \times f)_* \nu_{\Ga} = (f^{-1} \times \id)_*f_*\nu_{\Ga},
$$ the claim follows.
As $\sigma_2 \neq \sigma_1$, this contradicts (4), proving the implication  $(4) \Rightarrow (1)$.

\medskip
{\bf \noindent Second case:} The first case we just considered
implies  $$(2) \Rightarrow (1) \Leftrightarrow (3).$$

We now claim $(1)\Rightarrow (2)$. 
 Suppose that $\Gr$ is not Zariski dense. Then $\rank G_2 = 1$ and hence $\fa = \R^2$ by Lemma \ref{Zdense}. We use the same notation used in
 the proof of the implication $(4)\Rightarrow (1)$ of the first case. In particular, $\delta_\Ga=\delta_{\rho(\Ga)}$.
 Recall that we have in this situation that
 $\nu_{\sigma_1}=(f^{-1} \times \id)_* f_* \nu_{\Ga}$ is a $(\Gr, \sigma_2)$-conformal measure on $\La_{\rho}$. Therefore,
 Proposition \ref{prop.loabscont} implies that
 
  $$\ess_{\nu_{\sigma_1}}(\Gr) \subset \ker (\sigma_1 - \sigma_2) = \br (1,1)
 $$
where $\R(1, 1) :=  \{ t \cdot (1, 1) \in \R^2 : t \in \R\}$.
 It follows from  Proposition \ref{prop.jordaness} (note that we have not assumed that $\Ga_\rho$ is Zariski dense in that proposition) that $\ess_{\nu_{\sigma_1}}(\Gr)$ contains $\la(\Delta_{\rho}) - E$ for some Zariski dense Schottky subgroup $\Delta < \Ga$ and a finite subset $E$. Since $\rho$ induces an isometry $X_1 \to X_2$, we have $\la(\Delta_{\rho}) \subset \R(1, 1)$, and hence $\la(\Delta_{\rho}) - E$ generates a dense subgroup of $\R(1, 1)$  by Theorem \ref{dense0}. Consequently, $\ess_{\nu_{\sigma_1}}(\Gr) = \R(1, 1)$, as desired.

It remains to prove $(1) \Leftrightarrow (4)$. To prove $(4)\Rightarrow (1)$, suppose that $\Ga_\rho$ is Zariski dense.
By Corollary \ref{cor.manyconformal}, there exists a $(\Gr, \psi)$-conformal measure $\nu$ on $\La_{\rho}$ for a tangent linear form $\psi \neq \sigma_1$. By the first case, $[\nu] \neq [\nu_{\sigma_1}]$. Therefore (4) does not hold. This proves the claim.

We now show that (1) implies (4). Again, suppose that $\Gr$ is not Zariski dense in $G$, and hence $\rho$ extends to a Lie group isomorphism $G_1 \to G_2$. 
As before, since $\rho$ induces an isometry $X_1\to X_2$, we get that the limit cone 
$\L_{\rho}$ is equal to $\R_{\ge 0} (1, 1)$ and $\delta_{\Ga} = \delta_{\rho(\Ga)}$. We denote by $\delta = \delta_{\Ga} = \delta_{\rho(\Ga)}$; so the growth indicator is given by $\psi_{\rho}(t, t) = \delta t$ on $\L_{\rho}$.

Let $\psi \in \fa^*$ be a linear form tangent to $\psi_{\rho}$ with a $(\Gr, \psi)$-conformal measure $\nu_{\psi}$ on $\La_{\rho}$. 
We need to show that $[\nu_\psi]=[\nu_{\sigma_1}]$.
Let $\pi:\F\to \F_1$ denote the canonical projection to the first factor.
We first claim that
$\pi_* \nu_{\psi}$ is a $\Ga$-conformal measure of dimension $\delta$ on $\La$. Since $\psi(1,1)=\delta=\sigma_1(1,1)=\sigma_2(1,1)$,
we have $\psi = (1-c)\sigma_1 + c \sigma_2$ for some $c \in \R$.
As $\nu_\psi$ is $(\Ga_\rho, \psi)$-conformal,
we have for any $\xi \in \La$ and $\ga = (g, \rho(g)) \in \Gr$, $${ d\ga_* \nu_{\psi} \over d \nu_{\psi}}(\xi, f(\xi)) = e^{(1-c) \delta_{\Ga} \beta_{\xi}(o_1, g o_1) + c \delta_{\rho(\Ga)} \beta_{f(\xi)}(o_2, \rho(g) o_2)} = e^{\delta \beta_{\xi}(o_1, g o_1)}$$
where $o_2=\rho(o_1)$.
As in the proof of Proposition \ref{lem.pushforward}, this implies that for any $g \in \Ga$ and $\xi \in \La$, $${d g_* \pi_* \nu_{\psi} \over d \pi_* \nu_{\psi}}(\xi) = e^{ \delta \beta_{\xi}(o_1, go_1)},$$ proving the claim. Since $\Ga$ is of divergence type, it follows that
$\pi_* \nu_{\psi}$ is equivalent to $\nu_{\Ga}$  by Theorem \ref{thm.uniquepsdiv}. As $\nu_{\psi}$ is supported on $\La_{\rho} = (\id \times f)(\La)$ and $\pi$ is injective on $\La_{\rho}$ with $\pi|_{\La_\rho}^{-1} = \id \times f$, we have
$$\nu_{\psi} = (\id \times f)_*\pi_* \nu_{\psi} .$$
 Therefore, 
 $[\nu_{\psi}] = [(\id \times f)_*\nu_{\Ga}] = [\nu_{\sigma_1}]$. 
 This finishes the proof.
\end{proof}

\subsection*{Proof of Theorem \ref{mmm}}

Let $\nu_{\rho(\Ga)}$ be a $(\rho(\Ga), \psi)$-conformal measure on $\La_{\rho(\Ga)}$.
Then by Lemma \ref{s2},
$ \nu_{\sigma_2}: = (f^{-1} \times \id)_* \nu_{\rho(\Ga)}$ is $(\Gr, \sigma_2)$-conformal for
$\sigma_2 = \psi \circ \pi_2 \in \fa^*$, and
$$[\nu_{\sigma_2}] = [\nu_{\sigma_1}] \mbox{ if and only if } [\nu_{\rho(\Ga)}] = [f_*\nu_{\Ga}].$$

Hence if $[f_* \nu_{\Ga}] = [\nu_{\rho(\Ga)}]$, then $[\nu_{\sigma_1}] = [\nu_{\sigma_2}]$. From the first case of Theorem \ref{mplast}, $\Gr$ is not Zariski dense, which implies that $\rho$ extends to a Lie group isomorphism $G_1 \to G_2$ by Lemma \ref{Zdense}. Hence we get $(1) \Rightarrow (2)$.

Conversely, suppose that $\rho$ extends to a Lie group isomorphism $G_1 \to G_2$.
Then $\Ga_\rho$ is not Zariski dense and $\op{rank} G_2 = 1$. By the work of Patterson \cite{Patterson1976limit} and Sullivan \cite{Sullivan1979density}, there exists a $\rho(\Ga)$-conformal measure $\nu_{\rho(\Ga)}$ on $\La_{\rho(\Ga)}$ of dimension $\delta_{\rho(\Ga)}$.  We then consider $\sigma_2$ and $\nu_{\sigma_2}$ defined same as above. Since $\sigma_2$ is tangent to $\psi_{\rho}$ by Lemma \ref{s1}, it follows from the second case of Theorem \ref{mplast} that $[\nu_{\sigma_2}] = [\nu_{\sigma_1}]$; so $[\nu_{\rho(\Ga)}] = [f_* \nu_{\Ga}]$. This proves $(2) \Rightarrow (1)$.

\begin{Rmk} \rm
    When $\Ga$ is of divergence type, $\nu_{\Ga}$ is $\Ga$-ergodic and hence the first condition in Theorem \ref{mmm} is the same as saying that some $\rho(\Ga)$-conformal measure is absolutely continuous with respect to $f_*\nu_{\Ga}$ (see Lemma \ref{ergg}).
\end{Rmk}

\section{Measure class rigidity for Anosov representations}\label{sec.anosov}
One class of Zariski dense discrete subgroups of $G$ where the space of conformal measures on $\La_\Delta$ is well-understood is the class of Anosov subgroups with respect to $P$.
There are several equivalent definitions of Anosov subgroups (one is given in Example \ref{ex.bdr}(3)), and the following is due to
Kapovich-Leeb-Porti \cite{Kapovich2017anosov} (see also \cite{Labourie2006anosov}, \cite{Guichard2012anosov}).  Recall that $\Pi$ denotes the set of all simple roots of $\fg$ with respect to $\fa^+$:
\begin{Def}[Anosov representation] \label{ano} For a finitely generated group $\Sigma$ and a non-empty subset $\Pi_0\subset \Pi$,
a representation $\rho: \Sigma \to G$ is said to be Anosov with respect to $\Pi_0$
if for all $g \in \Sigma$ and 
for all $\alpha\in \Pi_0$,
$$ \alpha (\mu(\rho(g))) \ge C_1 |g| -C_2 $$  where $C_1, C_2>0$ are uniform constants and $|\cdot |$ is a word metric.
Its image $\Delta=\rho(\Sigma)$ is called an Anosov subgroup with respect to $\Pi_0$.
Anosov subgroups with respect to $P$ mean Anosov subgroups with respect to $\Pi$.
\end{Def}

 We recall the following theorem which follows from \cite[Theorem 1.3]{lee2020invariant}, together with \cite[Theorem 1.4]{edwards2021unique} and \cite[Theorem 1.2]{lee2022dichotomy}.

\begin{theorem}[{\cite[Theorem 1.3]{lee2020invariant}}]\label{loo}
Let $\Delta<G$ be a Zariski dense Anosov subgroup  with respect to $P$.
For each linear form $\psi$ tangent to $\psi_\Delta$, there exists a unique $(\Delta,\psi)$-conformal measure $\nu_\psi$ on $\F$ with respect to $o$.
Moreover, we have
\begin{enumerate}
    \item  the map $\psi\mapsto \nu_\psi$
gives a bijection between the space of all tangent linear forms and
    the space of all $\Delta$-conformal measures supported on $\La_\Delta$;
    \item  if $\psi_1\ne \psi_2$, then $\nu_{\psi_1}$ and $\nu_{\psi_2}$
    are  mutually singular to each other.
    \end{enumerate}
\end{theorem}    
We  mention that the space of all tangent linear forms in this case is homeomorphic to
$\br^{\rank G-1}$. See also
\cite{sambarino2022report} for some partial extension for Anosov subgroups with respect to general $\Pi_0$.

\subsection*{Proof of Theorem \ref{ananintro}}
We now begin the proof of Theorem \ref{ananintro}. Recall the notations introduced there:  let $\rho_i : \Sigma \to G_i$ a Zariski dense Anosov representation of a hyperbolic group $\Sigma$ into a simple real algebraic group $G_i$ for $i = 1, 2$. Recall the notations: $\Ga = \rho_1(\Sigma) < G_1$, $\rho = \rho_2 \circ \rho_1^{-1} : \Ga \to G_2$, and
$f=f_2\circ f_1^{-1} : \La_{\Ga} \to \F_2$ where  $f_i : \partial \Sigma \to \F_i$ is a
$\rho_i$-equivariant embedding for $i=1,2$. Let $\nu_{\Ga}$ and $\nu_{\rho(\Ga)}$ be $\Ga$-conformal and $\rho(\Ga)$-conformal measures on $\La_{\Ga}$ and $\La_{\rho(\Ga)}$ respectively.

 Let $G=G_1\times G_2$ and $\fa = \fa_1 \oplus \fa_2$. Let $\psi_1 \in \fa_1^*$ and $\psi_2 \in \fa_2^*$ be linear forms associated to $\nu_{\Ga}$ and $\nu_{\rho(\Ga)}$ respectively. For convenience, we put $\nu_{\psi_1} = \nu_{\Ga}$ and $\nu_{\psi_2} = \nu_{\rho(\Ga)}$. Let $\Psi_i = \psi_i \circ \pi_i$ for each $i = 1, 2$. Then $\Psi_i$ is a linear form on $\fa$ and it follows from Proposition \ref{lem.pushforward}  and Lemma \ref{s2} that $$(\id \times f)_* \nu_{\psi_1} \quad \mbox{and} \quad (f^{-1} \times \id)_*\nu_{\psi_2}$$ are $(\Gr, \Psi_1)$-conformal and $(\Gr, \Psi_2)$-conformal measures on $\La_{\rho}$ respectively.

Suppose that $\rho : \Ga \to G_2$ does not extend to a Lie group isomorphism $G_1 \to G_2$. Since $G_1$ and $G_2$ are simple, Lemma \ref{Zdense} implies that $\Gr$ is a Zariski dense Anosov subgroup of $G_1 \times G_2$. Hence by Theorem \ref{loo}, we have $$(\id \times f)_* \nu_{\psi_1} \perp (f^{-1} \times \id)_*\nu_{\psi_2}.$$ As in the proof of 
Lemma \ref{s2}, this implies that $$f_* \nu_{\psi_1} \perp \nu_{\psi_2}.$$
This proves Theorem \ref{ananintro}.

\subsection*{Proof of Corollary \ref{cor.hitchin}}
Suppose that for a Hitchin representation $\rho : \Ga \to \PSL_d(\R)$ and its boundary map  $f : \S^1 \to \F$, we have that $f_* \op{Leb}_{\S^1}$ is non-singular to a $\rho(\Ga)$-conformal measure. Since $\Ga < \PSL_2(\R)$ is cocompact, it is of divergence type and $\op{Leb}_{\S^1}$ is the $\Ga$-conformal measure on $\La_{\Ga} = \S^1$. Since a Hitchin representation is an Anosov representation, it follows from Theorem \ref{ananintro} that $\rho : \Ga \to \PSL_d(\R)$ extends to a Lie group isomorphism $\PSL_2(\R) \to \PSL_d(\R)$, which is equivalent to say that $d = 2$ and $\rho$ is a conjugation.

\begin{Rmk} \rm
We finally mention that replacing \cite[Theorem 10.20]{lee2020invariant}  by an analogous result in
\cite{sambarino2022report} for Anosov representations with respect to a general parabolic subgroup,
we can also prove an analogous statement to
Theorem \ref{anan} in those cases provided that the Furstenberg boundaries and limit sets are appropriately replaced as well.
\end{Rmk}


\end{document}